\documentclass[a4paper,reqno]{amsart}
\usepackage[english]{babel}
\usepackage{amsmath,amsthm,amssymb,amsfonts}
\usepackage{enumitem}

\usepackage{hyperref}

\numberwithin{equation}{section}

\allowdisplaybreaks

\DeclareMathOperator{\Span}{span}
\DeclareMathOperator{\supp}{supp}
\DeclareMathOperator*{\esssup}{ess\,sup}

\newtheorem{theorem}{Theorem}[section]
\newtheorem{lemma}[theorem]{Lemma}
\newtheorem{proposition}[theorem]{Proposition}
\newtheorem{corollary}[theorem]{Corollary}

\theoremstyle{definition}

\newcommand{\tc}{\,:\,}

\newcommand\TT{{\mathbb{T}}}
\newcommand\RR{{\mathbb{R}}}
\newcommand\CC{{\mathbb{C}}}
\newcommand\NN{{\mathbb{N}}}
\newcommand\ZZ{{\mathbb{Z}}}

\newcommand{\cI}{\mathcal{I}}          
\newcommand{\cS}{\mathcal{S}}          

\newcommand{\vell}{l}                 

\newcommand{\sfera}{\mathbb{S}}   
\newcommand{\equat}{\mathbb{E}}   

\newcommand{\opL}{\mathcal{L}}   
\newcommand{\opG}{\mathcal{G}}   

\newcommand{\dist}{\varrho}   
\newcommand{\weight}{\varpi}  
\newcommand{\meas}{\sigma}    

\newcommand{\gr}{\mathrm{gr}}

\newcommand{\bigO}{\mathcal{O}}    

\newcommand{\IS}{I_\sfera}      

\newcommand{\bE}{\mathbf{E}}    
\newcommand{\bM}{\mathbf{M}}

\newcommand{\Kern}{\mathcal{K}}  

\newcommand{\Sob}[2]{L^{#1}_{#2}}  

\newcommand{\LegF}{\mathrm{P}}     
\newcommand{\SpH}{Y}               
\newcommand{\tSpH}{\widetilde{Y}}  

\DeclareMathOperator{\sgn}{sgn}

\DeclareFontFamily{U}{matha}{\hyphenchar\font45}
\DeclareFontShape{U}{matha}{m}{n}{
      <5> <6> <7> <8> <9> <10> gen * matha
      <10.95> matha10 <12> <14.4> <17.28> <20.74> <24.88> matha12
      }{}
\DeclareSymbolFont{matha}{U}{matha}{m}{n}
\DeclareFontSubstitution{U}{matha}{m}{n}

\DeclareFontFamily{U}{mathx}{\hyphenchar\font45}
\DeclareFontShape{U}{mathx}{m}{n}{
      <5> <6> <7> <8> <9> <10>
      <10.95> <12> <14.4> <17.28> <20.74> <24.88>
      mathx10
      }{}
\DeclareSymbolFont{mathx}{U}{mathx}{m}{n}
\DeclareFontSubstitution{U}{mathx}{m}{n}

\DeclareMathDelimiter{\vvvert}{0}{matha}{"7E}{mathx}{"17}

\newcommand{\spnt}[2]{\left\lfloor #1, #2 \right\rceil}

\begin{document}

\title[A sharp multiplier theorem on the Grushin sphere]{From refined estimates for spherical harmonics to a sharp multiplier theorem on the Grushin sphere}
\author{Valentina Casarino}
\address{Universit\`a degli Studi di Padova\\Stradella san Nicola 3 \\I-36100 Vicenza \\ Italy}
\email{valentina.casarino@unipd.it}
\author{Paolo Ciatti}
\address{Universit\`a degli Studi di Padova\\Via Marzolo 9 \\I-35100 Padova \\ Italy}
\email{paolo.ciatti@unipd.it}
\author{Alessio Martini}
\address{School of Mathematics \\ University of Birmingham \\ Edgbaston \\ Birmingham \\ B15 2TT \\ United Kingdom}
\email{a.martini@bham.ac.uk}

\newcommand{\acknowledgments}{\section*{Acknowledgments}
The authors are very grateful to Fulvio Ricci for bringing this problem to their attention.

A substantial part of this work was developed while the third-named author was visiting the DTG of the Universit\`a di Padova as a recipient of a ``Visiting Scientist 2016'' grant; he gratefully thanks the Universit\`a di Padova for the support and ospitality. This work was partially supported by the EPSRC Grant ``Sub-Elliptic Harmonic Analysis'' (EP/P002447/1), the Progetto GNAMPA 2017 ``Analisi armonica e teoria spettrale di Laplaciani", the Progetto GNAMPA 2016 ``Calcolo funzionale per operatori subellittici su variet\`a'', and the Progetto PRIN 2015 ``Variet\`a reali e complesse: geometria, topologia e analisi armonica''. The authors are members of the Gruppo Nazionale per l'Analisi Matematica, la Probabilit\`a e le loro Applicazioni (GNAMPA) of the Istituto Nazionale di Alta Matematica (INdAM).}


\begin{abstract}
We prove a sharp multiplier theorem of Mihlin--H\"ormander type for the Grushin operator on the unit sphere in $\RR^3$, and a corresponding boundedness result for the associated Bochner--Riesz means. The proof hinges on precise pointwise bounds for spherical harmonics.
\end{abstract}

\maketitle

\section{Introduction}
Since its introduction in \cite{Grushin}, the so-called Grushin operator of step $r+1$, defined in the bidimensional setting as
\begin{equation}\label{eq:plane_grushin}
\opG_r=-\left(\frac{\partial}{\partial u}\right)^2-u^{2r} \left(\frac{\partial}{\partial v}\right)^2,
\end{equation}
where $r$ is a positive integer and $(u,v)$ are coordinates on $\RR^2$, received considerable attention as a prototypical example of a
degenerate elliptic, hypoelliptic operator.
The $(r+1)$-step Grushin plane, that is $\RR^2$ with the Lebesgue measure and the control distance  associated to $\opG_r$,
turns out to be  a space of homogeneous type in the sense of Coifman and Weiss, and many classical problems in harmonic analysis, like restriction estimates, spectral multipliers theorems, Hardy inequalities, heat kernel bounds, $L^p$ estimates for the wave equation, and properties of Hardy and BMO spaces
have been successfully addressed in the last years in this framework, also in higher dimension and for fractional values of $r$
(see, e.g., \cite{Garofalo,Meyer,RS, MSi, JST,CS,  JT, MMu, BFI, CO,  DzJ,LRY} and references therein).

Despite the prototypical nature of \eqref{eq:plane_grushin}, the study of analogous problems for Grushin-type operators on more general manifolds often proves to be challenging and not as many results are available. In a few recent works \cite{BFIuno, BPS, Pesenson}, some attention has been given to 2-step Grushin-type operators on the unit sphere
\[
\sfera=\{z\in \RR^3 \tc z_1^2 + z_2^2 + z_3^2 = 1 \}
\]
in $\RR^3$.
The operator studied in \cite{BPS} is self-adjoint with respect to a measure on the sphere that is singular along the equator $\equat = \{z \in \sfera \tc z_3 = 0\}$ and can be thought of the Laplace--Beltrami operator for a certain almost-Riemannian structure on $\sfera$ \cite{BoL}. Instead \cite{BFIuno,Pesenson} consider, in analogy with the analysis extensively carried out on the Grushin plane, an operator $\opL$ that is self-adjoint with respect to the standard rotation-invariant measure $\meas$ on $\sfera$.

Let us give some more details on the latter perspective, which is the one that we adopt in this paper.
The sphere $\sfera$ has a natural, rotation-invariant Riemannian structure induced by the ambient space $\RR^3$. The vector fields $Z_1,Z_2,Z_3$, defined by
\[
Z_1 = z_3 \frac{\partial}{\partial z_2} -z_2 \frac{\partial}{\partial z_3}  , \qquad
Z_2 = z_1 \frac{\partial}{\partial z_3} -z_3 \frac{\partial}{\partial z_1} , \qquad
Z_3 = z_2 \frac{\partial}{\partial z_1} -z_1 \frac{\partial}{\partial z_2} ,
\]
span at each point of $\sfera$ the tangent space to $\sfera$ and
\[
\Delta = -(Z_1^2+Z_2^2+Z_3^2)
\]
is the Laplace--Beltrami operator on $\sfera$. Since
$[Z_1,Z_2] = Z_3$,
the system of vector fields $\{Z_1,Z_2\}$ is bracket-generating and determines a sub-Riemannian structure on $\sfera$.
We call Grushin sphere the sphere $\sfera$, equipped with this sub-Riemannian structure and with the standard Riemannian measure $\meas$, and spherical Grushin operator the self-adjoint sub-elliptic operator
\[
\opL=-(Z_1^2+Z_2^2).
\]
Note that $\opL$ degenerates on the equator $\equat$, since $\Span\{Z_1,Z_2\}$ is one-dimensional there, whereas $\opL$ is elliptic on $\sfera \setminus \equat$. Indeed, in suitable spherical coordinates $(\theta,\varphi)$,
\begin{equation}\label{eq:sph_grushin_coords}
\opL f = -\frac{1}{\cos\theta} \frac{\partial}{\partial \theta} \left( \cos\theta \frac{\partial}{\partial \theta} f\right)
- {\tan^2 \theta}\,\left(\frac{\partial}{\partial \varphi}\right)^2 f,
\end{equation}
where the singularity $\theta = 0$ corresponds to the equator $\equat$.

The classical spherical harmonics $\SpH_{\ell,m}$ ($\ell\in \NN$, $-\ell\leq m\leq \ell$) with pole $(1,0,0)$, usually studied as eigenfunctions of the Laplace--Beltrami operator $\Delta$, are eigenfunctions of the Grushin operator $\opL$ as well:
\[
\opL \SpH_{\ell,m} = \lambda_{\ell,m} \SpH_{\ell,m},
\]
where $\lambda_{\ell,m} = \ell(\ell+1)-m^2$; therefore they can be used to describe the spectral decomposition and define a functional calculus for $\opL$.
Here we are interested in $L^p$ boundedness properties of the multiplier operator
\begin{equation}\label{eq:delFL}
F(\opL)= \sum_{(\ell,m) \in \NN \times \ZZ \tc |m| \leq \ell} F(\lambda_{\ell,m}) \,\pi_{\ell, m},
\end{equation}
initially defined on $L^2(\sfera)$, where $F : \RR \to \CC$ is a bounded Borel function
and $\pi_{\ell,m}$ is the orthogonal projection operator mapping $L^2(\sfera)$ onto $\CC \SpH_{\ell,m}$.
More precisely, we seek minimal differentiability requirements which, imposed on $F$, guarantee the boundedness of $F(\opL)$ on $L^p(\sfera)$ for $p \neq 2$. This is a classical problem in harmonic analysis, which has been studied for many operators and eigenfunction expansions, though sharp results are only known in a few cases.

Choose a nontrivial cut-off function $\eta \in C^\infty_c((0,\infty))$. For all $q \in [1,\infty]$ and $s \in [0,\infty)$, let $\Sob{q}{s}(\RR)$ denote the $L^q$ Sobolev space on $\RR$ of order $s$. Our main result is the following multiplier theorem of Mihlin--H\"ormander type for $\opL$.

\begin{theorem}\label{thm:mainmain}
Let $d = 2$ be the topological dimension of $\sfera$. Let $s > d/2$.
\begin{enumerate}[label=(\roman*)]
\item\label{en:mainmain_compact} For all continuous functions $F : \RR \to \CC$ supported in $[1/4,1]$,
\begin{equation}\label{eq:l1bdd}
\sup_{t>0} \| F(t \opL) \|_{L^1(\sfera) \to L^1(\sfera)} \leq C_s \, \|F\|_{\Sob{2}{s}}.
\end{equation}
\item\label{en:mainmain_mh} For all bounded Borel functions $F : \RR \to \CC$ such that $F|_{(0,\infty)}$ is continuous,
\begin{equation}\label{eq:mainmain_mh_wt11}
\|F(\opL) \|_{L^1(\sfera) \to L^{1,\infty}(\sfera)} \leq C_s \, \sup_{t \geq 0} \| F(t \cdot) \, \eta \|_{\Sob{2}{s}}.
\end{equation}
Hence, whenever the right-hand side of \eqref{eq:mainmain_mh_wt11} is finite, the operator $F(\sqrt{\opL})$ is of weak type $(1,1)$ and bounded on $L^p(\sfera)$ for all $p \in (1,\infty)$.
\item Let $p \in [1,\infty]$. If $\delta > (d-1)|1/2-1/p|$, then the Bochner--Riesz means $(1-t\opL)_+^\delta$ of order $\delta$ associated with $\opL$ are bounded on $L^p(\sfera)$ uniformly in $t \in (0,\infty)$.
\end{enumerate}
\end{theorem}

Note that the supremum in the right-hand side of \eqref{eq:mainmain_mh_wt11} includes $t=0$ and therefore majorizes $|F(0)|$; moreover, the finiteness of the supremum depends only on $F$ and not on the choice of the cut-off $\eta$. We remark that, for a Borel function $F : \RR \to \CC$, the finiteness of the right-hand side of \eqref{eq:l1bdd} implies, by Sobolev's embedding, that $F$ coincides Lebesgue-a.e.\ with a continuous function; the continuity assumption in part \ref{en:mainmain_compact} of Theorem \ref{thm:mainmain} is meant to specify that it is this continuous representative of $F$ that is to be used to define the operators $F(t\opL)$. A similar observation applies to the continuity assumption in part \ref{en:mainmain_mh}.

The above multiplier theorem is sharp, in the sense that the condition $s > d/2$ on the order of smoothness $s$ cannot be weakened. The optimality follows from a now standard transplantation argument, originally due to Mitjagin \cite{Mi}, mainly known from a paper by Kenig, Stanton and Tomas \cite{KST}.
Since $\opL$ is elliptic away from the equator, by means of Mitjagin's argument we deduce the sharpness of Theorem \ref{thm:mainmain} from the sharpness of the analogous results for the Laplace operator on $\RR^d$.

It should be noted that the analogue of Theorem \ref{thm:mainmain} where $\opL$ is replaced by the Laplace--Beltrami operator $\Delta$
follows from the classical results of \cite{SeeSo,Sogge4}, that hold more generally for elliptic operators on compact manifolds. On the other hand, the Grushin operator $\opL$ is not elliptic on the whole sphere, which makes its analysis particularly arduous, especially when it comes to sharp multiplier theorems.

Indeed the lack of ellipticity of the Grushin operator $\opL$ is related to the fact that the ``local dimension'' $Q$ associated to the sub-Riemannian structure at each point of the equator is larger \cite{FePh} than the topological dimension $d$ of the sphere (in this case $Q=3$), and it would be relatively straightforward \cite{He,DOS} to prove a weaker result than Theorem \ref{thm:mainmain}, with smoothness condition $s > Q/2$, and the $L^2$ Sobolev norms in \eqref{eq:l1bdd}-\eqref{eq:mainmain_mh_wt11} replaced by $L^\infty$ Sobolev norms.

The mismatch between the natural dimensional parameter associated to the geometry of a sub-elliptic operator and the optimal smoothness condition in a multiplier theorem of Mihlin--H\"ormander type was first noted in the case of a sub-Laplacian on the Heisenberg groups \cite{He_heis,MuSt}, and has been since then the object of a number of studies. Despite many recent developments, especially in the case of homogeneous sub-Laplacians on $2$-step stratified groups (see \cite{MMuGAFA} and references therein), the problem of determining the optimal smoothness condition remains widely open: sharp results are known only in particular cases, and there appears to be no optimal result of the same level of generality of those available for elliptic operators.
The present work can therefore be considered as part of an ongoing effort in the investigation of sub-elliptic operators, whereby progress on the general case is sought by analysing particularly significant instances.

In these respects, the spherical Grushin operator $\opL$ presents a combination of features that makes its study especially interesting. Indeed, differently from the aforementioned homogeneous sub-Laplacians, the operator $\opL$ has discrete spectrum, and moreover it is not group-invariant, in the sense that there is no transitive group action on $\sfera$ preserving $\opL$. Sharp multiplier theorems have already been obtained for sub-elliptic operators with discrete spectrum \cite{CoS,CoKS,CCMS,ACMMu}, but the known results always involve group-invariant operators. Conversely, \cite{MSi,MMu} give a sharp multiplier theorem for the $2$-step Grushin operator $\opG_1$ and its higher-dimensional versions, which are not group-invariant, but have continuous spectrum. The combination of discrete spectrum and lack of group-invariance is therefore a novel feature of the result discussed here, which is reflected in the technical challenges of its proof.

\medskip

More precisely, the proof of Theorem \ref{thm:mainmain} follows the general pattern used in previous works \cite{He_heis,MuSt,CoS,MSi}, which essentially reduces the sharp multiplier theorem to a so-called ``weighted Plancherel-type estimate''. In the case of the spherical Grushin operator $\opL$, the proof of the weighted Plancherel-type estimate requires very precise pointwise bounds for the spherical harmonics $\SpH_{\ell,m}$, which must be uniform with respect to the parameters $\ell$ and $m$.
From this point of view, our proof has some similarities with that of \cite{MSi} for the plane Grushin operator, where instead precise bounds for Hermite functions are exploited.

However, in the case of \cite{MSi}, the required bounds for Hermite functions were available in the literature and could be immediately used. Instead, perhaps surprisingly, despite the fact that spherical harmonics have ubiquitous applications in mathematics and in physics,
and a wealth of estimates for them have been obtained, also recently \cite{Kra,Ward,BDWZ,Haagerup,Han,Franck},
we could not find in the literature suitable pointwise bounds for $\SpH_{\ell,m}$ that could be used ``out of the box''.

Hence part of this paper is devoted to the proof of these bounds, which may be of independent interest.
Namely, since  the spherical harmonics $\SpH_{\ell,m}$ may be explicitly written in terms of the associated  Legendre functions $\LegF_{\ell}^{m}$,
we use some earlier asymptotic approximations for $\LegF_{\ell}^{m}$ \cite{BoydDunster, Olver75, Olver},
in combination with classical estimates for Hermite and Bessel functions \cite{AW,MuSp} and more recent estimates for spherical harmonics \cite{Ward,BDWZ,Haagerup},
to show that the behaviour of  $\SpH_{\ell,m}$ follows two different regimes depending on the range where $\ell$ and $m$ vary; this corresponds to the well-known fact that both Hermite functions and Bessel functions can be obtained as suitable limits of Legendre functions \cite[eqs. (5.6.3) and (8.1.1)]{Szego}.
More precisely, if $|m| \geq \epsilon (\ell+1/2)$ for some $\epsilon \in (0,1)$ (the ``Hermite regime''), in the light of some  estimates  in \cite{Olver75, Olver} we prove that
\[
| \SpH_{\ell,m}(z)| \leq C_\epsilon \begin{cases}
((1+\ell)^{-1} + |z_3^2-a_{\ell,m}^2|)^{-1/4} &\text{for all $z\in\sfera$,}\\
|z_3|^{-1/2} \exp(-c \ell z_3^2) &\text{if $|z_3| \geq K \, a_{\ell,m}$,}
\end{cases}
\]
where $K,c,C_\epsilon$ are suitable positive constants and the ``critical point'' $a_{\ell,m}$ is given by
\begin{equation}\label{eq:criticalpoints}
a_{\ell,m}=\sqrt{1-b_{\ell,m}^2}, \qquad b_{\ell,m} = |m|/(\ell+1/2).
\end{equation}
If instead $|m|\leq \epsilon (\ell+\frac{1}{2})$ (the ``Bessel regime''), from some bounds
proved in \cite{BoydDunster} we derive that
\[
| \SpH_{\ell,m}(z)| \leq C_\epsilon \begin{cases}
\left( \frac{(1+|m|)^{4/3}}{(1+\ell)^{2}} + |z_3^2-a_{\ell,m}^2|\right)^{-1/4} & \text{for all $z \in\sfera$},\\
b_{\ell,m}^{-1/2} \, 2^{-m} &\text{if $\sqrt{1-z_3^2} \leq b_{\ell,m}/4$.}
\end{cases}
\]

This dichotomy reflects two different types of eigenfunction concentration occurring on the sphere, which are best exemplified by zonal ($m=0$) and highest weight ($m=\ell$) spherical harmonics: zonal harmonics present extreme concentration at the poles, while highest weight harmonics are highly concentrated around the equator (cf.\ \cite{Sogge3,Sogge2}). These two competing concentration phenomena determine the form of the sharp $L^p \to L^2$ bounds of the spectral projections of the Laplace--Beltrami operator $\Delta$, which were first proved by Sogge \cite{Sogge1} and in the last thirty years have been extended and improved in several ways (see \cite{Sogge3,Ze1,Ze2} and references therein). Given the ``dual nature'' of the spherical Grushin operator $\opL$ (elliptic at the poles and sub-elliptic at the equator), it is natural that the interaction of the two types of concentration plays a substantial role in its analysis too.

As a matter of fact, the existence of two different regimes, as well as the presence of two discrete parameters instead of one, makes it more technically demanding, to deal with spherical harmonics, compared to the case of Hermite functions and the plane Grushin operator.
This is true not only in the derivation of the pointwise bounds, but also in their application to prove the weighted Plancherel-type estimate: one should compare the proofs of \cite[Lemma 9 and Proposition 10]{MSi} with those of Propositions \ref{prp:indici_alti}, \ref{prp:indici_bassi} and \ref{prp:weighted_plancherel} below.

The relation between the spherical and plane Grushin operators appears not only on the eigenfunctions' side.
Indeed, a comparison of \eqref{eq:plane_grushin} and \eqref{eq:sph_grushin_coords} highlights that the plane Grushin operator $\opG_1$ can be thought of as a ``local model'' for the spherical Grushin operator $\opL$ at each point of the equator. This can be made precise by means of a ``contraction'' procedure, whereby $\opG_1$ is written as a limit of rescaled versions of $\opL$ through nonisotropic dilations $\delta_t(u,v) = (tu,t^2 v)$ on $\RR^2$. A generalisation of Mitjagin's transplantation technique \cite[Section 5]{M} can then be applied to show that our Theorem \ref{thm:mainmain} for $\opL$ implies the corresponding result for $\opG_1$ proved in \cite{MSi}.

In view of the results on higher dimensional ``flat'' Grushin spaces $\RR^{d_1}\times\RR^{d_2}$  \cite{MSi, MMu},
it would be interesting to study spectral multipliers on higher dimensional Grushin-type spheres as well (see, e.g., \cite{BFIuno}).
In particular, a natural higher-dimensional generalisation of $\opL$ would be a spherical Grushin operator with
a one-codimensional equatorial singularity, whose local model is a flat Grushin operator with $d_1 = 1 < d_2$. In the flat case, however, the classical approach based on weighted Plancherel-type estimates appears not to suffice to obtain a sharp multiplier theorem when $d_1 < d_2$ \cite{MSi} and, to overcome this difficulty, a different approach was developed in \cite{MMu}.
It is likely that similar issues may arise in the analysis of higher-dimensional spherical Grushin operators and that new techniques and ideas may be necessary.
We hope to extend our methods to cover this different framework in the near future.

\medskip

The schema of the paper is as follows.
In Section \ref{s:preliminaries} we recall the main facts about the unit sphere $\sfera$
in $\RR^3$ with the Grushin structure. We also discuss the metric features of $\sfera$,
providing a precise estimate for the sub-Riemannian distance.
Finally, we discuss the decomposition of $L^2 (\sfera)$ in terms of spherical harmonics and introduce the orthogonal polynomials and the special functions
that play a role in our study.
In Section \ref{s:bounds} we prove pointwise bounds for the spherical harmonics  $\SpH_{\ell,m}$, uniformly valid with respect to $\ell,m$.
Section \ref{s:weighted} is devoted to the proof of the weighted Plancherel-type estimate. The latter
is the gist in the proof of the sharp multiplier theorem, which is  presented in Section \ref{s:multiplier}.

\medskip

In the following, we use the ``variable constant convention'', according to which positive constants are denoted by symbols such as $C$ or $C_\epsilon$,
and these are not necessarily equal at different occurrences.
For any two nonnegative quantities $a$ and $b$ we write $a\lesssim b$  instead of $a \leq C b$ and $a \gtrsim b$ instead of $a \geq C b$; moreover $a \simeq b$ stands for the conjunction of $a \gtrsim b$ and $a \lesssim b$.

\section{Preliminaries}\label{s:preliminaries}

\subsection{The Grushin sphere}\label{ss:sfera}

Let the vector fields $Z_1,Z_2,Z_3$ on the unit sphere $\sfera$ in $\RR^3$ be defined as in the introduction, and recall that
$\Delta = -(Z_1^2+Z_2^2+Z_3^2)$
is the Laplace--Beltrami operator on $\sfera$. Note that
\[
[Z_1,Z_2] = Z_3, \qquad [Z_2,Z_3] = Z_1, \qquad [Z_3,Z_1] = Z_2.
\]
So the system of vector fields $\{Z_1,Z_2\}$ is $2$-step bracket-generating and determines a sub-Riemannian structure on $\sfera$ (more on sub-Riemannian geometry can be found, e.g., in \cite{BellaicheRisler,CalinChang,Montgomery}). At each point $z \in \sfera$, the horizontal distribution
\[
H_z \sfera = \Span \{Z_1|_z, Z_2|_z\}
\]
is given the inner product $\langle \cdot,\cdot \rangle_z$ corresponding to the norm
\[
|v|_z = \inf\{\sqrt{a^2+b^2} \tc a,b \in \RR, \, v = a Z_1|_z + b Z_2|_z\}
\]
for all $z \in \sfera$ and $v \in H_z\sfera$. Note that the horizontal distribution has not constant rank and degenerates at the equator $\equat = \{z \in \sfera \tc z_3 = 0\}$. If $z \in \equat$, then $\dim H_z \sfera = 1$ and $|\cdot|_z$ coincides with (the restriction of) the standard Riemannian norm. If $z \in \sfera \setminus \equat$, then $\dim H_z \sfera = 2$ and $\{Z_1|_z,Z_2|_z\}$ is an orthonormal basis of $H_z\sfera$ with respect to the inner product $\langle \cdot,\cdot \rangle_z$.

We denote by $\dist$ the corresponding sub-Riemannian distance on $\sfera$: in other words, for all $z,z' \in \sfera$, $\dist(z,z')$ is the infimum of the lengths $L(\gamma)$ of all horizontal curves $\gamma$ joining $z$ to $z'$. Here a horizontal curve is an absolutely continuous curve $\gamma : [a,b] \to \sfera$ whose derivative $\gamma'(t) \in H_{\gamma(t)}\sfera$ for almost all $t \in [a,b]$; the length of such a curve is defined by $L(\gamma) = \int_a^b |\gamma'(t)|_{\gamma(t)} \,dt$.

The sphere $\sfera$, with the above sub-Riemannian structure and the standard Riemannian measure $\meas$, will be referred to as the Grushin sphere. Note that our choice of the measure on $\sfera$ differs from that of \cite{BPS}. In particular, with our choice of $\meas$, the vector fields $Z_1,Z_2$ are (formally) skew-adjoint and the corresponding ``sum of squares'' operator $\opL=-(Z_1^2+Z_2^2)$,
called the Grushin operator, is self-adjoint. Note that the sub-Riemannian distance $\dist$ is the ``control distance'' associated to $\opL$ and finite propagation speed holds: for all $t \in \RR$ and $f,g \in L^2(\sfera)$,
\[
\langle \cos(t \sqrt{\opL}) f, g \rangle = 0 \qquad\text{whenever } \dist(\supp f,\supp g) > |t|
\]
(see \cite{Melrose}; cf.\ also \cite{CM} and \cite[Proposition 4.1]{RS}).

Observe that
\[
\left[\Delta, Z_j\right]=0 \qquad \text{for } j=1,2,3.
\]
In particular, if we define
\[
T=iZ_3,
\]
then the operators $\Delta,T,\opL$ are self-adjoint and commute pairwise, and moreover
\[
\opL = \Delta - T^2.
\]
This shows that the spectral decomposition of $\opL$ can be reduced to the joint spectral decomposition of $\Delta,T$.

Similarly as in \cite{BPS},
if we introduce on $\sfera$
the coordinate system
\[
\spnt{\theta}{\varphi} := (\cos \theta \cos \varphi, \cos \theta \sin \varphi, \sin \theta),
\]
where $\theta\in[-\frac\pi 2, \frac\pi2]$ and $\varphi \in \TT := \RR/2\pi\ZZ$, then
\[
d\meas(\spnt{\theta}{\varphi}) = \cos \theta \,d\theta \,d\varphi
\]
and
\[
Z_1 =\cos\varphi \,X_2-\sin\varphi \,X_1, \qquad
Z_2 =\sin\varphi \,X_2+\cos\varphi \,X_1, \qquad
Z_3 =-\frac{\partial}{\partial \varphi},
\]
where $X_1$ and $X_2$ are defined by
\begin{equation}\label{def:campi}
X_1=\frac{\partial}{\partial \theta}, \qquad
X_2=\tan \theta\frac{\partial}{\partial \varphi}.
\end{equation}
In particular
\begin{equation}\label{eq:opLpolarcoords}
\opL = Z_1^+ Z_1 +Z_2^+ Z_2 = X_1^+ X_1 + X_2^+ X_2
\end{equation}
(here $X_j^+$ and $Z_j^+$ denote the formal adjoints of $X_j$ and $Z_j$), which gives \eqref{eq:sph_grushin_coords}, 
and moreover
\[
T= -i \frac{\partial}{\partial \varphi}.
\]

The expression \eqref{eq:sph_grushin_coords} for the spherical Grushin operator in coordinates highlights its similarity to the plane Grushin operator $\opG_1$ of \eqref{eq:plane_grushin}. Hence it is not surprising that, analogously as for the Grushin plane (see, e.g., \cite[Proposition 5.1]{RS}), a precise estimate for the distance $\dist$ can be obtained.

\begin{proposition}\label{prp:subriemannian}
The sub-Riemannian distance $\dist$ on $\sfera$ satisfies
\begin{equation}\label{eq:rho-dist}
\dist (\spnt{\theta}{\varphi},\spnt{\theta'}{\varphi'}) \simeq |\theta-\theta'| +
\min\left\{{|\varphi-\varphi'|}^{1/2}, \frac{|\varphi-\varphi'|}{\max\{|\tan \theta|,|\tan \theta'|\}}\right\},
\end{equation}
where $|\varphi-\varphi'|$ denotes the arclength distance on $\TT$ between $\varphi,\varphi'$. In particular, if $\theta=\theta'=0$, then
\begin{equation}\label{eq:rho-dist-zero}
\dist (\spnt{0}{\varphi},\spnt{0}{\varphi'}) \simeq {|\varphi-\varphi'|}^{1/2};
\end{equation}
moreover, for all $\varepsilon > 0$, if $\max \{ |\theta|, |\theta'|\}\geq \varepsilon$ then
\begin{equation}\label{eq:rho-dist-far}
\dist (\spnt{\theta}{\varphi},\spnt{\theta'}{\varphi'}) \simeq \dist_R (\spnt{\theta}{\varphi},\spnt{\theta'}{\varphi'}),
\end{equation}
where $\dist_R$ is the Riemannian distance on the sphere $\sfera$ and the implicit constants may depend on $\varepsilon$.
Consequently, the $\meas$-measure $V(\spnt{\theta}{\varphi},r)$ of the $\dist$-ball centred at $\spnt{\theta}{\varphi}$ with radius $r \geq 0$ satisfies
\begin{equation}\label{eq:volume}
V(\spnt{\theta}{\varphi}, r) \simeq \min\{ 1, r^2 \max \{r, |\theta| \} \}.
\end{equation}
\end{proposition}

From the above estimates it is clear that $\dist$ is topologically, but not Lipschitz equivalent to the Riemannian distance $\dist_R$ on the whole $\sfera$. Indeed, the measure of a $\dist$-ball with centre on the equator $\equat$ is given by $V(\spnt{0}{\varphi},r) \simeq r^3$ for $r$ small, while the corresponding volume of a Riemannian ball behaves as $r^2$. However, far from the equator $\equat$, $\dist$
is actually equivalent to $\dist_R$.

To prove Proposition \ref{prp:subriemannian}, it is convenient first to introduce a couple of lemmas. The first one reduces global equivalence to local equivalence via compactness.

\begin{lemma}\label{lem:loc_glob_eq}
Let $X$ be a Hausdorff topological space and $K \subseteq X \times X$ be compact. Let $\Phi, \Psi : K \to [0,\infty)$ be continuous functions with the following properties:
\begin{itemize}
\item each of $\Phi$ and $\Psi$ is \emph{point-separating}, i.e., for all $(x,y) \in K$, if $\Phi(x,y) = 0$ or $\Psi(x,y) = 0$ then $x=y$;
\item $\Phi$ and $\Psi$ are \emph{locally equivalent}, i.e., for all $p \in X$ with $(p,p) \in K$ there exist a neighbourhood $U$ of $p$ in $X$ and a positive constant $C$ such that
\begin{equation}\label{eq:equivalence}
C^{-1} \Phi(x,y) \leq \Psi(x,y) \leq C \Phi(x,y)
\end{equation}
for all $(x,y) \in (U \times U) \cap K$.
\end{itemize}
Then $\Phi,\Psi$ are \emph{globally equivalent}, i.e., there exist a positive constant $C$ such that \eqref{eq:equivalence} holds for all $(x,y) \in K$.
\end{lemma}
\begin{proof}
Since $X$ is Hausdorff, the diagonal $\Delta_X$ of $X \times X$ is closed. For all $p \in X$ with $(p,p) \in K$, let $U_p$ be an open neighbourhood of $p$ in $X$ such that \eqref{eq:equivalence} holds for all $(x,y) \in (U_p\times U_p) \cap K$ with a positive constant $C_p$ instead of $C$. By compactness of $K \cap \Delta_X$, we can find finitely many points $p_1,\dots,p_k$ so that $U_{p_1}\times U_{p_1},\dots,U_{p_k} \times U_{p_k}$ cover $K \cap \Delta_X$. Then $K' := K \setminus \bigcup_{i=1}^k U_{p_i} \times U_{p_i}$ is compact and does not intersect $\Delta_X$, so $\Phi,\Psi$ are strictly positive on $K'$ and, by continuity and compactness, there exists a positive constant $C'$ such that $(C')^{-1} < \Phi(x,y), \Psi(x,y) < C'$ for all $(x,y) \in K$. By taking $C = \max \{(C')^2, C_{p_1},\dots,C_{p_k}\}$, the conclusion follows.
\end{proof}

The second lemma reduces local equivalence of sub-Riemannian distances to an infinitesimal condition, i.e., the equivalence of the corresponding norms on the horizontal distributions.

\begin{lemma}\label{lem:subriemannian-enhanced}
Let $M, N$ be sub-Riemannian manifolds, with sub-Riemannian distance functions $d_M,d_N$ respectively. Let $F: U \to V$ be a diffeomorphism between open sets $U\subseteq M$ and $V\subseteq N$ and $C_1,C_2$ be positive constants
such that, for all $p\in U$,
\begin{equation}\label{eq:equalitydF}
dF_p (H_p M)=H_{F(p)}N
\end{equation}
and, for all $v\in H_p M$,
\begin{equation}\label{eq:enhanced}
C_1 |dF_p (v)|\leq |v|\leq C_2 |dF_p (v)|.
\end{equation}
Then for all points $p\in U$ there is a neighbourhood $\Omega\subseteq U$ of $p$ such that
\[
C_1 \, d_N(F(x), F(y)) \leq d_M(x,y) \leq C_2 \, d_N(F(x), F(y)).
\]
for all $x,y\in \Omega$.
\end{lemma}
\begin{proof}
By \eqref{eq:equalitydF} and \eqref{eq:enhanced}, for every horizontal curve $\gamma$ in $U$, $F\circ \gamma$ is a horizontal curve in $V$ and their lengths $L(\gamma)$ and $L(F \circ \gamma)$ are related by
\begin{equation}\label{eq:rel_lengths}
C_1 L(F\circ \gamma) \leq L(\gamma) \leq C_2 L(F \circ \gamma).
\end{equation}

For all $p\in M$, there is $r>0$ such that
the sub-Riemannian ball $B_M(p,r)$ is contained in $U$. Then, for all $x,y \in B_M (p, r/2)$,
\begin{equation}\label{eq:dist_M}
d_M (x,y)=\inf \{ L(\gamma): \gamma \text{ horizontal curve in $U$ joining $x$ to $y$} \},
\end{equation}
since a curve joining $x$ to $y$ and leaving $U$ must have length at least $r$. Similarly, if  $r'>0$ is such that $B_N(F(p),r') \subseteq V$, then, for all $x,y \in B_N (F(p), r'/2)$,
\begin{equation*}
d_N (x,y)= \inf \{ L(\gamma): \gamma \text{ horizontal curve in $V$ joining $x$ to $y$} \}.
\end{equation*}

Take $\Omega = B_M (p, r/2) \cap F^{-1} (B_N (F(p), r'/2))$. Then, for all $x,y \in \Omega$,
\[\begin{split}
d_N (F(x),F(y))&=
\inf \{
L(\gamma):
\gamma \text{ horizontal curve in $V$ joining $F(x)$ to $F(y)$}
\}\\
&=
\inf \{
L(F\circ \gamma):
\gamma \text{ horizontal curve in $U$ joining $x$ to $y$}
\},
\end{split}\]
where we used the fact that $F : U \to V$ is a diffeomorphism preserving the horizontal vectors. This, combined with \eqref{eq:rel_lengths} and \eqref{eq:dist_M}, gives the conclusion.
\end{proof}

\begin{proof}[Proof of Proposition \ref{prp:subriemannian}]
Let us first prove the equivalence \eqref{eq:rho-dist-far}. Note that the sub-Riemannian distance $\dist$ and the Riemannian distance $\dist_R$ are locally equivalent far from the equator $\equat$: indeed, since $H_p M = T_p M$ for all $p \in \sfera \setminus \equat$ and the Riemannian and sub-Riemannian inner products on $T_p M$ depend continuously on $p$, we can apply Lemma \ref{lem:subriemannian-enhanced} with $M$ and $N$ being the Riemannian and sub-Riemannian $\sfera$ respectively, and $F$ being the identity map restricted to any open subset $U$ of $\sfera$ whose closure does not intersect $\equat$. The equivalence \eqref{eq:rho-dist-far} then follows from Lemma \ref{lem:loc_glob_eq} applied with
\begin{equation}\label{eq:semi_far_equator}
K = \{ (\spnt{\theta}{\varphi},\spnt{\theta'}{\varphi'}) \in \sfera \times \sfera \tc \max\{|\theta|,|\theta'|\} \geq \varepsilon \};
\end{equation}
indeed the local equivalence condition in Lemma \ref{lem:loc_glob_eq} need only be tested at those points $p \in \sfera$ such that $(p,p) \in K$, i.e., far from the equator.

We now prove the equivalence \eqref{eq:rho-dist}. Note that the expression in the right-hand side of \eqref{eq:rho-dist} defines a continuous function $\Phi : \sfera \times \sfera \to [0,\infty)$, which is point-separating in the sense of Lemma \ref{lem:loc_glob_eq}. Hence, in order to prove the equivalence \eqref{eq:rho-dist}, it is enough to show that $\Phi$ and $\dist$ are locally equivalent.

We first show that $\dist$ and $\Phi$ are locally equivalent at each point of $\sfera \setminus \equat$. For this it is certainly enough to prove that, for every $\varepsilon > 0$, the functions $\dist$ and $\Phi$ are globally equivalent on the set $K$ defined in \eqref{eq:semi_far_equator}.
Now, if $\max\{|\theta|,|\theta'|\} \geq \varepsilon$, then $1/\max\{|\tan\theta|,|\tan\theta'|\} \simeq \min\{\cos\theta,\cos\theta'\}$
and $|\theta-\theta'| \simeq |\cos\theta-\cos\theta'|$, so
\[\begin{split}
\Phi(\spnt{\theta}{\varphi},\spnt{\theta}{\varphi'})
&\simeq |\cos\theta-\cos\theta'| + \min\{\cos\theta,\cos\theta'\} |\varphi-\varphi'| \\
&\simeq |(\cos\theta) \, e^{i\varphi} - (\cos\theta') \, e^{i\varphi'}| \\
&\simeq \dist_R(\spnt{\theta}{\varphi},\spnt{\theta}{\varphi'}) \\
&\simeq \dist(\spnt{\theta}{\varphi},\spnt{\theta}{\varphi'}),
\end{split}\]
where the implicit constants may depend on $\varepsilon$. In the last two steps we have used \eqref{eq:rho-dist-far} and the equivalence
$\dist_R(\spnt{\theta}{\varphi},\spnt{\theta}{\varphi'}) \simeq |(\cos\theta) \, e^{i\varphi} - (\cos\theta') \, e^{i\varphi'}|$
between the Riemannian distance of $\spnt{\theta}{\varphi}$ and $\spnt{\theta}{\varphi'}$ on $\sfera$ and the Euclidean distance of their projections on the equatorial plane,
which are both valid when $\max\{|\theta|,|\theta'|\} \geq \varepsilon$.

We now prove the local equivalence of $\dist$ and $\Phi$ at points of the equator $\equat$. Let $G : (-\pi/2,\pi/2) \times \RR \to \sfera$ be defined by
\[
G(u,v) = \spnt{u}{v}
\]
(by a small abuse of notation, we are not distinguishing $v \in \RR$ from its equivalence class in $\TT=\RR/2\pi\ZZ$).
Then $G$ is a local diffeomorphism, and moreover from \eqref{def:campi} it follows that
\begin{equation}\label{eq:Gsp_campi}
dG\left(\frac{\partial}{\partial u}\right)=X_1, \qquad
dG\left(\tan u \, \frac{\partial}{\partial v}\right)=X_2.
\end{equation}
Note that the H\"ormander system of vector fields $\left\{\frac{\partial}{\partial u}, u \, \frac{\partial}{\partial v}\right\}$ induces the sub-Riemannian structure of the Grushin plane $\RR^2$ (the corresponding sub-Riemannian distance will be denoted by $\dist_{\gr}$), while the system $\{X_1,X_2\}$ induces the sub-Riemannian structure of the Grushin sphere $\sfera$ away from the poles.
Since, for all $\varepsilon >0$, we have that $\tan u \simeq u$ if $|u| \leq \pi/2-\varepsilon$, from \eqref{eq:Gsp_campi} it follows that Lemma \ref{lem:subriemannian-enhanced} can be applied with $M$ being the Grushin plane $\RR^2$, $N$ being the Grushin sphere $\sfera$ and $F$ being the restriction of $G$ to a sufficiently small open neighbourhood of any point of $(-\pi/2,\pi/2) \times \RR$, thus obtaining that $\dist_{\gr}$ and $\dist \circ G$ are locally equivalent on $(-\pi/2,\pi/2) \times \RR$. In particular, every point of $\equat = G(\{0\} \times \RR)$ has a sufficiently small neighbourhood $U$ such that, for all $\spnt{\theta}{\varphi},\spnt{\theta'}{\varphi'} \in U$,
\[\begin{split}
\dist(\spnt{\theta}{\varphi},\spnt{\theta'}{\varphi'})
& \simeq \dist_{\gr}((\theta,\varphi),(\theta',\varphi')) \\
& \simeq |\theta-\theta'| + \min\left\{|\varphi-\varphi'|^{1/2}, \frac{| \varphi-\varphi'|}{|\theta|+|\theta'|} \right\} \\
& \simeq \Phi(\spnt{\theta}{\varphi},\spnt{\theta'}{\varphi'}),
\end{split}\]
where the estimate for $\dist_{\gr}$ from \cite[Proposition 5.1]{RS} was used. This shows the local equivalence of $\dist$ and $\Phi$ at each point of $\equat$, and concludes the proof of \eqref{eq:rho-dist}.

The estimate \eqref{eq:volume} for the volume of balls is easily obtained from the previous estimates for $\dist$, by considering separately the cases $|\theta| \leq \pi/4$ and $|\theta| \geq \pi/4$.
\end{proof}

\subsection{Spherical harmonics and spectral decompositions}

Let
\[
\IS = \{ (\ell,m) \in \NN \times \ZZ \tc |m| \leq \ell\}.
\]
In the following, for all $(\ell,m) \in \IS$, the symbol $\SpH_{\ell,m}$ shall denote a classical spherical harmonic on the sphere $\sfera$, explicitly given by (cf.\ \cite[eq.\ (6.10.7)]{Zw})
\begin{equation}\label{eq:spher-harm}
\SpH_{\ell,m} (\spnt{\theta}{\varphi})=\sqrt{\frac{2\ell+1}{4\pi} \,\frac{(\ell-m)!}{(\ell+m)!}} \, e^{im\varphi} \, \LegF_{\ell}^{m}(\sin \theta),
\end{equation}
for $\theta \in [-\frac{\pi}{2}, \frac{\pi}{2}]$ and $\varphi \in[0, 2\pi]$.
Here $\LegF_{\ell}^{m}$ denotes the associated Legendre function (also known as Ferrers function) of indices $\ell \in \NN$ and $m \in \ZZ$, with $|m| \leq \ell$, defined as follows:
for all $\ell,m \in \NN$ with $m \leq \ell$,
\begin{equation}\label{eq:Plm}
\LegF_{\ell}^{m}(x)= (-1)^m (1-x^2)^{m/2} \left(\frac{d}{dx}\right)^m P_{\ell}(x),
\end{equation}
and
\begin{equation}\label{eq:Plmmeno}
\LegF_{\ell}^{-m}(x)= (-1)^m\frac{(\ell-m)!}{(\ell+m)!} \,\LegF_{\ell}^m (x).
\end{equation}
In \eqref{eq:Plm}
$P_\ell$  denotes the Legendre polynomial of degree $\ell \in \NN$,
given  by
\begin{equation}\label{eq:Legendre}
P_\ell (x)=P^{(0,0)}_\ell (x),
\end{equation}
$P^{(0,0)}_\ell$ being a Jacobi polynomial of degree $\ell$ and indices both equal to $0$.
More generally, by the symbol
$P^{(\alpha,\beta)}_k$ we shall  denote the Jacobi polynomial of degree $k \in \NN$
and indices $\alpha,\beta>-1$, defined by means of Rodrigues' formula as
\begin{equation}\label{eq:jacobi}
P^{(\alpha,\beta)}_k(x)= \frac{(-1)^k}{2^k\, k!} (1-x)^{-\alpha} (1+x)^{-\beta} \left(\frac{d}{dx}\right)^k \left((1-x)^{\alpha+k} (1+x)^{\beta+k} \right)
\end{equation}
for $x \in (-1,1)$. Note that, for all $\ell,m\in \NN$ with $m \leq \ell$,
\begin{equation}\label{eq:LegendreJacobi}
\LegF_\ell^m(x) = (-1)^m \frac{(\ell+m)!}{2^m \ell!} (1-x^2)^{m/2} P_{\ell-m}^{(m,m)}(x).
\end{equation}
(cf.\ \cite[\S 10.10, eq.\ (46)]{EMOT}). For more details about orthogonal polynomials we refer the reader to the classical book by Szeg\H{o}
\cite{Szego}.

For future convenience, we collect some well-known properties of
$\SpH_{\ell,m}$.

\begin{lemma}\label{lem:SteinWeiss}
The following identities hold.
\begin{enumerate}[label=(\roman*)]
\item\label{en:sphharm_orthon}
For all $(\ell,m),(\ell',m') \in \IS$
\[
\int_\sfera \SpH_{\ell,m} (z) \, \overline{\SpH_{\ell',m'} (z)} \,d\meas(z) = \delta_{\ell\ell'} \delta_{mm'}.
\]
\item\label{en:sphharm_sym} For all $(\ell,m) \in \IS$ and $z \in \sfera$,
\[
\SpH_{\ell,-m}  (z) = (-1)^m\overline{\SpH_{\ell,m} (z)} .
\]
\item\label{en:sphharm_sum}
For all $\ell \in \NN$ and $z \in \sfera$,
\[
\sum_{m=-\ell}^{\ell} \left|\SpH_{\ell,m}(z)\right|^2= \meas(\sfera)^{-1}\, (2\ell+1).
\]
\end{enumerate}
\end{lemma}
\begin{proof}
\ref{en:sphharm_orthon} We refer to  \cite[eq.\ (6.10.9)]{Zw}.

\ref{en:sphharm_sym}. This follows from \eqref{eq:Plmmeno}; see also \cite[eq.\ (6.10.8)]{Zw}.

\ref{en:sphharm_sum} See \cite[Ch. 4, Corollary 2.9]{Stein-Weiss}.
\end{proof}

The  spherical harmonics form an orthonormal basis of the Hilbert space of square-integrable functions $L^2(\sfera)$ \cite[\S IV.2]{Stein-Weiss}, i.e., every such function $f$ can be expressed as a linear combination of spherical harmonics
\begin{equation}\label{eq:base}
f=\sum_{(\ell,m)\in\IS} c_{\ell,m} \, \SpH_{\ell,m} ,
\end{equation}
where the coefficients $c_{\ell,m}$ may
be computed as
\begin{equation}\label{eq:coefficienti}
c_{\ell,m}=\langle f, \SpH_{\ell,m}  \rangle = \int_\sfera f(z)\,\overline{\SpH_{\ell,m} (z) } \,d\meas(z).
\end{equation}

We recall moreover the classical spectral decompositions
\[
\Delta \SpH_{\ell,m} = \ell (\ell+1) Y_{\ell,m}
\]
and
\[
T \SpH_{\ell,m} = m \SpH_{\ell,m},
\]
for all $(\ell,m) \in \IS$, whence
\[
\opL \SpH_{\ell,m} = \lambda_{\ell,m} \SpH_{\ell,m},
\]
where
\begin{equation}\label{eq:eigen}
\begin{split}
\lambda_{\ell,m} &= \ell(\ell+1)-m^2 \\
&=(\ell-m+1/2)(\ell+m+1/2)-1/4.
\end{split}
\end{equation}

We also recall that, for all $\ell \in \NN$, the function
\begin{equation*}
\SpH_{\ell,0} (\spnt{\theta}{\varphi}) = \sqrt{\frac{2\ell+1}{4\pi}} P_{\ell}(\sin \theta)
\end{equation*}
is the zonal spherical harmonic of degree $\ell$, while
\begin{equation}\label{eq:vettore-peso-massimo}
\begin{split}
\SpH_{\ell,\ell} (\spnt{\theta}{\varphi})
&= \sqrt{\frac{2\ell+1}{4\pi} \,\frac{1}{(2\ell)!}} e^{i\ell\varphi} \LegF_{\ell}^{\ell}(\sin \theta) \\
&= \frac{(-1)^{\ell}}{2^\ell \, \ell!} \sqrt{\frac{2\ell+1}{4\pi} \,{(2\ell)!}} \,e^{i\ell\varphi} \, (\cos\theta)^{\ell}
\end{split}
\end{equation}
is the highest weight spherical harmonic of degree $\ell$.

The described spectral decomposition, together with the explicit formulas for $\opL$ and $T$ given in Section \ref{ss:sfera}, allow us to easily prove the following crucial estimate.

\begin{lemma}\label{lem:weighted_est}
Let $f \in L^2(\sfera)$ be orthogonal to $\SpH_{\ell,0}$ for all $\ell \in \NN$. Then, for all $\alpha \in [0,1]$,
\begin{equation}\label{eq:weighted_est}
\int_{\sfera} |\tan \theta|^{2\alpha} |f(\spnt{\theta}{\varphi})|^2 \,d\meas(\spnt{\theta}{\varphi}) \leq \| \opL^{\alpha/2} |T|^{-\alpha} f \|_{L^2(\sfera)}^2.
\end{equation}
\end{lemma}
\begin{proof}
Note that
\[
\| \opL^{\alpha/2} |T|^{-\alpha} f \|_{L^2(\sfera)}^2 = \sum_{\substack{(\ell,m) \in \IS \\ m \neq 0}} \lambda_{\ell,m}^{\alpha} |m|^{-2\alpha} |c_{\ell,m}|^2,
\]
where the $c_{\ell,m}$ are given by \eqref{eq:coefficienti}. Hence the inequality \eqref{eq:weighted_est} can be interpreted as a two-weighted bound for the linear operator $(c_{\ell,m})_{\ell,m} \mapsto \sum_{\ell,m} c_{\ell,m} \SpH_{\ell,m}$ and, by the Stein--Weiss theorem on interpolation with change of measure \cite{StW_interp}, it is enough to prove the inequality when $\alpha = 0$ and $\alpha = 1$.

On the other hand, if $\alpha = 0$, then we even have equality in \eqref{eq:weighted_est}.

Suppose instead that $\alpha = 1$. Note that we can write $f = T g$, where
\[
g = \sum_{\substack{(\ell,m) \in \IS \\ m \neq 0}} \frac{c_{\ell,m}}{m} \SpH_{\ell,m}.
\]
Hence
\begin{multline*}
\| \opL^{1/2} |T|^{-1} f \|_{L^2(\sfera)}^2 = \langle \opL g, g \rangle \\
= \|X_1 g\|_{L^2(\sfera)}^2 + \|X_2 g\|_{L^2(\sfera)}^2 \geq \int_{\sfera} |(\tan\theta) \, f(\spnt{\theta}{\varphi})|^2 \,d\meas(\spnt{\theta}{\varphi}),
\end{multline*}
where we have used the fact that $\opL= X_1^+ X_1 + X_2^+ X_2$ and $X_2 = i (\tan\theta) T$.
\end{proof}

\section{Pointwise and uniform bounds for spherical harmonics}\label{s:bounds}

We collect in this Section a number of uniform pointwise estimates for the spherical harmonics $\SpH_{\ell,m}$, that will be crucial in the remainder of the paper. Since $|\SpH_{\ell,m}(\spnt{\theta}{\varphi})|$ does not depend on $\varphi$, the required estimates are more conveniently expressed in terms of the functions $\tSpH_{\ell,m} : [-1,1] \to \RR$ defined by
\begin{equation}\label{eq:tildeY}
\tSpH_{\ell,m}(x) = \sqrt{\frac{2\ell+1}{4\pi} \,\frac{(\ell-m)!}{(\ell+m)!}} \, \LegF_{\ell}^{m}(x).
\end{equation}
Note that $|\SpH_{\ell,m}(\spnt{\theta}{\varphi})| = |\tSpH_{\ell,m}(\sin\theta)|$ and $\|\SpH_{\ell,m}\|_\infty = \|\tSpH_{\ell,m}\|_\infty$. Moreover, by Lemma \ref{lem:SteinWeiss}\ref{en:sphharm_sym},
\begin{equation}\label{eq:Yindexparity}
|\tSpH_{\ell,m}(x)| = |\tSpH_{\ell,-m}(x)|.
\end{equation}

We start with a summary of some  uniform weighted bounds that are available in the literature.

\begin{proposition}\label{prp:stime-classiche}
There exists a positive constant $C$ such that:
\begin{enumerate}[label=(\roman*)]
\item\label{en:stime_unif}
$\|\tSpH_{\ell,m}\|_\infty
\le
C(1+\ell)^{1/2}$ for all $(\ell,m) \in \IS$.
\item\label{en:stime_cos}
$|\tSpH_{\ell,m} (x)|\, (1-x^2)^{1/4}
\le C\,(1+\ell)^{1/4}$
for all $(\ell,m) \in \IS$ and $x \in [-1,1]$.
\item\label{en:stime_cossin}
$|\tSpH_{\ell,m} (x)|\, |x (1-x^2)|^{1/6} \le C \,(1+\ell)^{1/6}$
for all $(\ell,m) \in \IS$ and $x \in [-1,1]$.
\end{enumerate}
\end{proposition}
\begin{proof}
\ref{en:stime_unif}. This is an immediate consequence of Lemma \ref{lem:SteinWeiss}\ref{en:sphharm_sum}.

\ref{en:stime_cos}. This is derived in \cite[Proposition 6]{Ward} from estimates of \cite{Kra} on Jacobi polynomials, and is independently proved in \cite{Haagerup}.

\ref{en:stime_cossin} See \cite[Theorem 1]{BDWZ}.
\end{proof}

In the remainder of this Section we are going to prove some more refined estimates. Here and subsequently,
$a_{\ell, m}$ and $b_{\ell,m}$ will denote the numbers in $[0,1]$ defined by \eqref{eq:criticalpoints}.

We consider first the range of $(\ell,m) \in \IS$ where $|m| \geq \epsilon (\ell+1/2)$ for some $\epsilon \in (0,1)$. In this range, Olver \cite{Olver75,Olver} obtained a uniform asymptotic approximation of Legendre functions for large $\ell$ in terms of Hermite functions.

We recall that, for all $\nu \in \NN$, the $\nu$-th Hermite function is defined by
\begin{equation}\label{eq:Hermite-def}
h_\nu (x)= (-1)^{\nu} (\nu! \, 2^\nu \sqrt{\pi})^{-1/2} e^{x^2/2} \left(\frac{d}{dx}\right)^\nu e^{-x^2}.
\end{equation}
We shall need the following asymptotic properties.
\begin{lemma}\label{lem:Hermite-Th}
Set $N=2\nu+1$.
Then
\begin{equation*}
\left|h_\nu (x)\right| \leq \begin{cases}
C(N^{1/3}+|x^2 -N|)^{-1/4} &\text{ for all $x\in\RR$,}\\
C\exp (-c x^2) &\text{ for $x^2\geq 2N$.}
\end{cases}
\end{equation*}
\end{lemma}
\begin{proof}
See \cite[Theorem B]{AW} or \cite[Lemma 1.5.1]{Th}.
\end{proof}

As a consequence of Olver's approximation, we obtain the following estimates for spherical harmonics.

\begin{proposition}\label{prp:est_olver}
There exist constants $K \in [2,\infty)$ and $c \in (0,1)$ such that, for all $\epsilon \in (0,1)$, there exist $C_\epsilon$ such that, for all $(\ell,m) \in \IS$ with
$|m| \geq \epsilon (\ell+1/2)$,
\begin{equation}\label{eq:est_olver}
|\tSpH_{\ell,m}(x)| \leq C_\epsilon \begin{cases}
((1+\ell)^{-1} + |x^2-a_{\ell,m}^2|)^{-1/4} &\text{for all $x\in[-1,1]$,}\\
|x|^{-1/2} \exp(-c \ell x^2) &\text{for $|x| \geq K \, a_{\ell,m}$.}
\end{cases}
\end{equation}
\end{proposition}
\begin{proof}
Thanks to the various parity properties of spherical harmonics, it is enough to prove the above estimate when $m > 0$, $x \in [0,1)$ and $\tSpH_{\ell,m}$ is replaced by $\tSpH_{\ell,-m}$.

According to \cite[eqs.\ (3.4)-(3.6)]{Olver}, the following asymptotic approximation holds:
\begin{equation}\label{eq:olver_approx}
\begin{split}
\LegF_\ell^{-m}(x) &= \kappa_{\ell,m} \left(\frac{\zeta_{\ell,m}(x)^2-\alpha_{\ell,m}^2}{x^2-a_{\ell,m}^2}\right)^{1/4} \\
& \times [ U(-(\ell-m+1/2),\zeta_{\ell,m}(x) \sqrt{2\ell+1}) \\
&+ \bE^{-1} \bM(-(\ell-m+1/2),\zeta_{\ell,m}(x) \sqrt{2\ell+1}) \, \bigO(\ell^{-2/3})]
\end{split}
\end{equation}
for $\ell \to \infty$, uniformly in $x\in[0,1)$, $(\ell,m) \in \IS$ and $m \geq \epsilon (\ell+1/2)$. Here $U$ is the parabolic cylinder function (see, e.g., \cite[\S 5]{Olver75}), $\bE^{-1} \bM$ denotes the pointwise ratio of the auxiliary functions $\bM$ and $\bE$ defined in \cite[\S 5.8]{Olver75}, the numbers $\kappa_{\ell,m}$ and $\alpha_{\ell,m}$ are given by
\[
\alpha_{\ell,m} = \sqrt{2 \frac{\ell-m+1/2}{\ell+1/2}},\qquad
\kappa_{\ell,m} = \frac{(\ell+1/2)^{-1/4} \, 2^{-\frac{\ell+m}{2}}}{\Gamma(\frac{\ell+m}{2} + \frac{3}{4})},
\]
and $\zeta_{\ell,m} : [0,1) \to [0,\infty)$ is the increasing bijection satisfying $\zeta_{\ell,m}(a_{\ell,m}) = \alpha_{\ell,m}$ and implicitly defined by
\begin{align}
\int_{\alpha_{\ell,m}}^{\zeta_{\ell,m}(x)} (\tau^2-\alpha_{\ell,m}^2)^{1/2} \,d\tau &= \int_{a_{\ell,m}}^x \frac{(t^2-a_{\ell,m}^2)^{1/2}}{1-t^2}{ dt} \qquad\text{($a_{\ell,m} \leq x < 1$),} \label{eq:def_zeta_upper}\\
\int^{\alpha_{\ell,m}}_{\zeta_{\ell,m}(x)} (\alpha_{\ell,m}^2-\tau^2)^{1/2} \,d\tau &= \int^{a_{\ell,m}}_x \frac{(a_{\ell,m}^2-t^2)^{1/2}}{1-t^2}{ dt} \qquad\text{($0 \leq x \leq a_{\ell,m}$)}
\end{align}
(see \cite[eqs.\ (2.7)-(2.10)]{Olver}; note that in \cite{Olver} the symbols $\zeta$, $a$, $\alpha$ are used in place of our $\zeta_{\ell,m}$, $a_{\ell,m}$, $\alpha_{\ell,m}$).

By \eqref{eq:tildeY}, we can rewrite the approximation \eqref{eq:olver_approx} in terms of $\tSpH_{\ell,-m}$ as follows:
\begin{equation}\label{eq:olver_approx_Y}
\begin{split}
|x^2 - a_{\ell,m}^2|^{1/4} \tSpH_{\ell,-m}(x) &= \tilde \kappa_{\ell+m} \frac{|(\zeta_{\ell,m}(x)\sqrt{\ell+1/2})^2-2(\ell-m+1/2)|^{1/4}}{\sqrt{(\ell-m)!}} \\
& \times [ U(-(\ell-m+1/2),\zeta_{\ell,m}(x) \sqrt{2\ell+1}) \\
&+ \bE^{-1} \bM(-(\ell-m+1/2),\zeta_{\ell,m}(x) \sqrt{2\ell+1}) \, \bigO(\ell^{-2/3})],
\end{split}
\end{equation}
where, for all $k \in \NN$,
\[
\tilde \kappa_{k} = \sqrt{\frac{k!}{2\pi}} \, \frac{2^{-\frac{k}{2}}}{\Gamma(\frac{k}{2} + \frac{3}{4})} \simeq 1
\]
by Stirling's approximation.

We now show that the right-hand side of \eqref{eq:olver_approx_Y} is uniformly bounded (in absolute value). Note that from \cite[\S 5.8]{Olver75} it follows easily that
\[
|U| \leq \bE^{-1} \bM \leq \bM
\]
pointwise. Therefore it is enough to show that
\[
\frac{|z^2-N|^{1/4}}{\sqrt{\Gamma(N/2+1/2)}} \, \bM(-N/2,z \sqrt{2})
\]
is uniformly bounded for $z \in [0,\infty)$ and $N \in [1,\infty)$ (consider the substitutions $N = 2\ell-2m+1$, $z= \zeta_{\ell,m}(x) \sqrt{\ell+1/2}$). This follows from the estimate for $\bM$ given in \cite[eq.\ after (6.12)]{Olver75} and applied with $\mu = \sqrt{N}$ and $y = z/\sqrt{N}$: indeed from that estimate we deduce that
\[
\frac{|z^2-N|^{1/2} \bM^2(-N/2,z\sqrt{2})}{\Gamma(N/2+1/2)} \leq \kappa \frac{N^{1/3} |\eta|^{1/2}}{1+N^{1/3} |\eta|^{1/2}} \leq \kappa
\]
for some universal constant $\kappa$ and some $\eta \in \RR$ depending on $N$ and $z$ (see \cite[eq.\ (5.14)]{Olver75} for the definition of $\eta$ as a function of $y$ and $\mu$).

Since the right-hand side of \eqref{eq:olver_approx_Y} is uniformly bounded, we deduce that
\begin{equation}\label{eq:olver_bound_Y}
|\tSpH_{\ell,-m}(x)| \leq C_\epsilon |x^2-a_{\ell,m}^2|^{-1/4}
\end{equation}
uniformly in $x\in[0,1)$, $(\ell,m) \in \IS$ and $m \geq \epsilon(\ell+1/2)$. Note now that $a_{\ell,m} \leq 1-\delta_\epsilon = \sqrt{1-\epsilon^2}$ for $m \geq \epsilon(\ell+1/2)$. In particular from \eqref{eq:olver_bound_Y} it follows that
\[
|\tSpH_{\ell,-m}(x)| \leq C_\epsilon
\]
uniformly in $x\in[1-\delta_\epsilon/2,1)$, $(\ell,m) \in \IS$ and $m \geq \epsilon(\ell+1/2)$. On the other hand, from Proposition \ref{prp:stime-classiche}\ref{en:stime_cos} it follows that
\[
|\tSpH_{\ell,-m}(x)| \leq C_\epsilon (1+\ell)^{1/4}
\]
uniformly in $x\in[0,1-\delta_\epsilon/2]$ and $(\ell,m) \in \IS$. By combining the last two estimates, we obtain that
\[
|\tSpH_{\ell,-m}(x)| \leq C_\epsilon (1+\ell)^{1/4}
\]
uniformly in $x\in[0,1)$, $(\ell,m) \in \IS$ and $m \geq \epsilon(\ell+1/2)$. This estimate can in turn be combined with \eqref{eq:olver_bound_Y} to give the first inequality in \eqref{eq:est_olver}.

We now turn to the second inequality in \eqref{eq:est_olver}. For this we need a better control of the ``error term'' $\bE^{-1} \bM$ in \eqref{eq:olver_approx_Y}. Note that, according to \cite[\S 5.8]{Olver75}, one has
\begin{equation}\label{eq:olver_error_fine_control}
\bE^{-1} \bM (-b,w) = \sqrt{2} U(-b,w)
\end{equation}
for all $b \in [0,\infty)$ and $w \in [\bar\rho(b),\infty)$, where $\bar\rho : [0,\infty) \to [0,\infty)$ is a continuous function satisfying
\[
\bar\rho(b) = \sqrt{4b} + \bigO(b^{-1/6})
\]
for $b \to \infty$ (please note that $\bar\rho(b)$ here corresponds to $\rho(-b)$ in the notation of \cite{Olver75}). In particular there exists a constant
$\widetilde K \in [1,\infty)$ such that
\[
\bar\rho(b) \leq \widetilde K \sqrt{4b}
\]
for all $b \in [1/2,\infty)$ and therefore the identity \eqref{eq:olver_error_fine_control} holds for all $b \in [1/2,\infty)$ and $w \in [\widetilde K\sqrt{4b},\infty)$. If we take $b = \ell-m+1/2$, then we see that $\sqrt{4b} = \alpha_{\ell,m} \sqrt{2\ell+1}$; consequently
\[
U(-(\ell-m+1/2),\zeta_{\ell,m}(x) \sqrt{2\ell+1}) =
\frac{1}{\sqrt{2}} \bE^{-1} \bM(-(\ell-m+1/2),\zeta_{\ell,m}(x) \sqrt{2\ell+1})
\]
whenever $\zeta_{\ell,m}(x) \geq \widetilde K \alpha_{\ell,m}$, and therefore from \eqref{eq:olver_approx_Y} we deduce that
\begin{equation}\label{eq:olver_better_approx_Y}
\begin{split}
|x^2 - a_{\ell,m}^2|^{1/4} |\tSpH_{\ell,-m}(x)| &\leq C_\epsilon \frac{|(\zeta_{\ell,m}(x)\sqrt{\ell+1/2})^2-2(\ell-m+1/2)|^{1/4}}{\sqrt{(\ell-m)!}} \\
& \times | U(-(\ell-m+1/2),\zeta_{\ell,m}(x) \sqrt{2\ell+1}) |
\end{split}
\end{equation}
uniformly in $(\ell,m) \in \IS$, $m \geq \epsilon(\ell+1/2)$ and $x \in [0,1)$ such that $\zeta_{\ell,m}(x) \geq \widetilde K \alpha_{\ell,m}$.

Note now that, for all $\nu \in \NN$ and $z \in \RR$,
\[
U(-(\nu+1/2),z \sqrt{2}) = (\nu! \sqrt{\pi})^{1/2} h_\nu(z)
\]
(compare \cite[eq.\ (7.22)]{Te} and \cite[eq.\ (1.1.2)]{Th} with \eqref{eq:Hermite-def} above). Hence \eqref{eq:olver_better_approx_Y} can be rewritten as
follows:
\begin{equation}\label{eq:olver_hermite_approx_Y}
\begin{split}
|x^2 - a_{\ell,m}^2|^{1/4} |\tSpH_{\ell,-m}(x)| &\leq C_\epsilon |(\zeta_{\ell,m}(x)\sqrt{\ell+1/2})^2-(2(\ell-m)+1)|^{1/4} \\
& \times | h_{\ell-m}(\zeta_{\ell,m}(x) \sqrt{\ell+1/2}) |
\end{split}
\end{equation}
uniformly in $(\ell,m) \in \IS$, $m \geq \epsilon(\ell+1/2)$ and $x \in [0,1)$ such that $\zeta_{\ell,m}(x) \geq \widetilde K \alpha_{\ell,m}$. Moreover, if $\zeta_{\ell,m}(x) \geq \sqrt{2} \widetilde K \alpha_{\ell,m}$, then $(\zeta_{\ell,m}(x) \sqrt{\ell+1/2})^2 \geq 2 (2(\ell-m)+1)$; hence, by combining \eqref{eq:olver_hermite_approx_Y} and Lemma \ref{lem:Hermite-Th}, we deduce that there exists a constant $c \in (0,\infty)$ such that
\begin{equation}\label{eq:olver_exp_approx_Y}
\begin{split}
|x^2 - a_{\ell,m}^2|^{1/4} |\tSpH_{\ell,-m}(x)|
&\leq C_\epsilon |(\zeta_{\ell,m}(x)\sqrt{\ell+1/2})^2-(2(\ell-m)+1)|^{1/4} \\
& \times \exp(-2c(\ell+1/2) \zeta_{\ell,m}(x)^2)  \\
&\leq C_\epsilon \exp(-c\ell \zeta_{\ell,m}(x)^2)
\end{split}
\end{equation}
uniformly in $(\ell,m) \in \IS$, $m \geq \epsilon(\ell+1/2)$ and $x \in [0,1)$ such that $\zeta_{\ell,m}(x) \geq \sqrt{2} \widetilde K \alpha_{\ell,m}$.

Note that
\[
a_{\ell,m} \leq \alpha_{\ell,m} \leq \sqrt{2} a_{\ell,m}.
\]
We now claim that
\begin{equation}\label{eq:zeta_x}
\zeta_{\ell,m}(x) \geq x \qquad\text{for all } x \geq a_{\ell,m}.
\end{equation}
Assuming this claim, we see that, if $x \geq 2 \widetilde K a_{\ell,m}$, then
\[
\zeta_{\ell,m}(x) \geq 2 \widetilde K a_{\ell,m} \geq \sqrt{2} \widetilde K \alpha_{\ell,m},
\]
and moreover
\[
x^2 - a_{\ell,m}^2 \geq x^2/2.
\]
Hence from \eqref{eq:olver_exp_approx_Y} we deduce that
\[\begin{split}
|x|^{1/2} |\tSpH_{\ell,-m}(x)|
&\leq 2^{1/4} |x^2 - a_{\ell,m}^2|^{1/4} |\tSpH_{\ell,-m}(x)| \\
&\leq C_\epsilon \exp(-c\ell \zeta_{\ell,m}(x)^2) \\
&\leq C_\epsilon \exp(-c\ell x^2)
\end{split}\]
uniformly in $(\ell,m) \in \IS$, $m \geq \epsilon(\ell+1/2)$ and $x \in [0,1)$ such that $x \geq 2 \widetilde K a_{\ell,m}$; this yields the second inequality in \eqref{eq:est_olver}, with $K = 2 \widetilde K$.

We are left with the proof of \eqref{eq:zeta_x}. Indeed, by \eqref{eq:def_zeta_upper},
\[
\int_{\alpha_{\ell,m}}^{\zeta_{\ell,m}(x)} (\tau^2-\alpha_{\ell,m}^2)^{1/2} \,d\tau = \int_{a_{\ell,m}}^x \frac{(t^2-a_{\ell,m}^2)^{1/2}}{1-t^2} \,dt \geq \int_{\alpha_{\ell,m}}^{\max\{x,\alpha_{\ell,m}\}} (t^2-\alpha_{\ell,m}^2)^{1/2} \,dt,
\]
since $\alpha_{\ell,m} \geq a_{\ell,m}$, and \eqref{eq:zeta_x} follows.
\end{proof}

In the range $|m| \leq \epsilon(\ell+1/2)$, the behaviour of spherical harmonics is different and a uniform asymptotic expression is available \cite{BoydDunster} in terms of Bessel functions. Recall that the Bessel function of the first kind $J_\nu$ of order $\nu \in (-1,\infty)$ is given by
\[
J_\nu(z) = \sum_{m=0}^\infty \frac{(-1)^m (z/2)^{\nu+2m}}{m! \, \Gamma(m+\nu+1)}.
\]
From the power series development it is immediate to obtain the following bound.

\begin{lemma}\label{lem:besselestimate}
For all $\nu \in [-1/2,\infty)$ and $z \in \RR$,
\[
|J_\nu(z)| \leq \frac{|z/2|^\nu}{\Gamma(\nu+1)}
\]
\end{lemma}

As a consequence of Boyd and Dunster's approximation, we obtain the following estimates for spherical harmonics.

\begin{proposition}\label{prp:est_boyddunster}
Let $\epsilon \in (0,1)$. For all $(\ell,m) \in \IS$ with $|m| \leq \epsilon (\ell+1/2)$,
\begin{equation}\label{eq:est_boyddunster}
|\tSpH_{\ell,m}(x)|
\leq C_\epsilon \begin{cases}
\left( \frac{(1+|m|)^{4/3}}{(1+\ell)^{2}} + |x^2-a_{\ell,m}^2|\right)^{-1/4} & \text{for all $x \in[-1,1]$,}\\
b_{\ell,m}^{-1/2} \, 2^{-m} &\text{if $\sqrt{1-x^2} \leq b_{\ell,m}/4$.}
\end{cases}
\end{equation}
\end{proposition}
\begin{proof}
Similarly as in the proof of Proposition \ref{prp:est_olver}, we restrict to $x \in [0,1)$ and $m \geq 0$ and prove the inequalities \eqref{eq:est_boyddunster} with $\tSpH_{\ell,-m}$ in place of $\tSpH_{\ell,m}$.

According to \cite[eqs.\ (4.9), (3.6) and eq.\ below (3.11), applied with $n=0$]{BoydDunster}, the following asymptotic approximation holds:
\begin{equation}\label{eq:boyddunster_approx}
\begin{split}
\LegF_{\ell}^{-m}(x) &= \varkappa_{\ell,m} \left(\frac{\zeta_{\ell,m}(x)-b_{\ell,m}^2}{a_{\ell,m}^2-x^2}\right)^{1/4} \\
&\times \left[J_{m}((\ell+1/2) \,\zeta_{\ell,m}(x)^{1/2}) + E_m^{-1} M_m((\ell+1/2) \,\zeta_{\ell,m}(x)^{1/2}) \, \bigO(\ell^{-1})\right]
\end{split}
\end{equation}
as $\ell \to \infty$, uniformly in $x \in [0,1]$ and $(\ell,m) \in \IS$ with $0 \leq m \leq \epsilon(\ell+1/2)$, where the number $\varkappa_{\ell,m}$ is given by
\[
\varkappa_{\ell,m} = e^m \frac{(\ell+1/2-m)^{(\ell+1/2-m)/2}}{(\ell+1/2+m)^{(\ell+1/2+m)/2}},
\]
(see \cite[eq.\ (4.11)]{BoydDunster}), $E_m^{-1} M_m$ is the pointwise ratio of the auxiliary functions $M_m$ and $E_m$ defined in \cite[\S 3]{BoydDunster} (see also \cite[\S 12.1.3]{Olver-libro}), and $\zeta_{\ell,m} : [0,1] \to [0,\zeta_{\ell,m}(0)]$ is the decreasing bijection satisfying $\zeta_{\ell,m}(a_{\ell,m}) = b_{\ell,m}^2$ and implicitly defined by
\begin{align}
\int_{b_{\ell,m}^2}^{\zeta_{\ell,m}(x)} \frac{(\xi-b_{\ell,m}^2)^{1/2}}{2\xi} \,d\xi &= \int_x^{a_{\ell,m}} \frac{(a_{\ell,m}^2-s^2)^{1/2}}{1-s^2} \,ds \qquad\text{($0 \leq x \leq a_{\ell,m}$),} \label{eq:def_BDzeta_upper}\\
\int_{\zeta_{\ell,m}(x)}^{b_{\ell,m}^2} \frac{(b_{\ell,m}^2-\xi)^{1/2}}{2\xi} \,d\xi &= \int_{a_{\ell,m}}^x \frac{(s^2-a_{\ell,m}^2)^{1/2}}{1-s^2} \,ds \qquad\text{($a_{\ell,m} \leq x \leq 1$).} \label{eq:def_BDzeta_lower}
\end{align}
Note that  in \cite{BoydDunster} the symbols $\zeta$, $\alpha$, $c_{1,1}$ are used instead of our $\zeta_{\ell,m}$, $b_{\ell,m}$, $\varkappa_{\ell,m}$. Note also that, according to \cite[\S 3]{BoydDunster}, the approximation \eqref{eq:boyddunster_approx} holds uniformly provided $b_{\ell,m}$ and $\zeta_{\ell,m}(0)$ range in compact subsets of $[0,1)$ and $[0,\infty)$ respectively; these conditions are clearly satisfied under the assumption $0 \leq m \leq \epsilon(\ell+1/2)$, because $0 \leq b_{\ell,m} \leq \epsilon < 1$ and
moreover, by \eqref{eq:def_BDzeta_upper},
\begin{multline*}
\int_{1}^{\max\{1,\zeta_{\ell,m}(0)\}} \frac{(\xi-1)^{1/2}}{2\xi} \,d\xi \leq \int_{b_{\ell,m}^2}^{\zeta_{\ell,m}(0)} \frac{(\xi-b_{\ell,m}^2)^{1/2}}{2\xi} \,d\xi \\
= \int_0^{a_{\ell,m}} \frac{(a_{\ell,m}^2 -s^2)^{1/2}}{1-s^2} \,ds \leq \int_0^{1} (1 -s^2)^{-1/2} \,ds = \pi/2 = \int_{1}^{\bar\zeta} \frac{(\xi-1)^{1/2}}{2\xi} \,d\xi
\end{multline*}
for some $\bar\zeta \in (0,\infty)$ not depending on $\ell,m$, so $\zeta_{\ell,m}(0) \in [0,\bar\zeta]$.

By \eqref{eq:tildeY} we can rewrite \eqref{eq:boyddunster_approx} in terms of spherical harmonics as follows:
\begin{equation}\label{eq:boyddunster_Yapprox}
\begin{split}
|x^2-a_{\ell,m}^2|^{1/4} \, \tSpH_{\ell,-m}(x) &= \tilde \varkappa_{\ell,m} |(\ell+1/2)^2 \zeta_{\ell,m}(x)-m^2|^{1/4} \\
&\times \bigl[J_{m}((\ell+1/2) \,\zeta_{\ell,m}(x)^{1/2}) \\
&+ E_m^{-1} M_m((\ell+1/2) \,\zeta_{\ell,m}(x)^{1/2}) \, \bigO(\ell^{-1}) \bigr]
\end{split}
\end{equation}
uniformly in $x \in [0,1]$ and $(\ell,m) \in \IS$ with $0 \leq m \leq \epsilon(\ell+1/2)$,
where
\[
\tilde \varkappa_{\ell,m} = \sqrt{\frac{1}{2\pi} \,\frac{(\ell+m)!}{(\ell-m)!}} e^m \frac{(\ell+1/2-m)^{(\ell+1/2-m)/2}}{(\ell+1/2+m)^{(\ell+1/2+m)/2}} \simeq 1
\]
uniformly in $(\ell,m) \in \IS$ by Stirling's approximation.

From \cite[\S 12.1.3]{Olver-libro} it is clear that
\[
|J_\nu| \leq E^{-1}_\nu M_\nu \leq M_\nu
\]
pointwise for all $\nu \in [0,\infty)$. Moreover, by \cite[Appendix B, Lemma 2]{BoydDunster}, the quantity
\[
|z^2-\nu^2|^{1/4} M_\nu(z)
\]
is uniformly bounded for $z,\nu \in [0,\infty)$. Hence, by taking $\nu=m$ and $z=(\ell+1/2) \,\zeta_{\ell,m}(x)^{1/2}$, from \eqref{eq:boyddunster_Yapprox} we deduce that the bound
\begin{equation}\label{eq:boyddunster_bound_Y}
|x^2-a_{\ell,m}^2|^{1/4} \, |\tSpH_{\ell,-m}(x)| \leq C_\epsilon
\end{equation}
holds uniformly in $x \in [0,1]$ and $(\ell,m) \in \IS$ with $0 \leq m \leq \epsilon(\ell+1)$.

In order to complete the proof of the first inequality in \eqref{eq:est_boyddunster}, it is enough to show that
\begin{equation}\label{eq:est_boyddunster_unif}
|\tSpH_{\ell,-m}(x)| \leq C_\epsilon (1+\ell)^{1/2} (1+m)^{-1/3}
\end{equation}
for all $(\ell,m) \in \IS$ with $0 \leq m \leq \epsilon(\ell+1/2)$ and $x \in [0,1]$.
Note that this estimate is certainly true for $m=0$ by Proposition \ref{prp:stime-classiche}\ref{en:stime_unif}, hence we may assume $m>0$.

Let $y=(1-x^2)^{1/2}$ and note that \eqref{eq:boyddunster_bound_Y} can be rewritten as
\begin{equation}\label{eq:boyddunster_bound_Y_bis}
|\tSpH_{\ell,-m}(x)| \leq C_\epsilon \, |y^2-b_{\ell,m}^2|^{-1/4}.
\end{equation}
If $y \geq (1+\epsilon)/2$, then \eqref{eq:boyddunster_bound_Y_bis} implies that
\[
|\tSpH_{\ell,-m}(x)| \leq C_\epsilon \leq C_\epsilon (1+\ell)^{1/2} (1+m)^{-1/3}.
\]
Similarly, if $y \leq b_{\ell,m}/2$, then \eqref{eq:boyddunster_bound_Y_bis} implies that
\[
|\tSpH_{\ell,-m}(x)| \leq C_\epsilon b_{\ell,m}^{-1/2} = C_\epsilon \left(\frac{\ell+1/2}{m}\right)^{1/2} \leq C_\epsilon (1+\ell)^{1/2} (1+m)^{-1/3}.
\]
Finally, if $b_{\ell,m}/2 \leq y \leq (1+\epsilon)/2$, then, by Proposition \ref{prp:stime-classiche}\ref{en:stime_cossin},
\[
|\tSpH_{\ell,-m}(x)| \leq C_\epsilon \, b_{\ell,m}^{-1/3} (1+\ell)^{1/6} = C_\epsilon (1+\ell)^{1/2} m^{-1/3}
\]
and \eqref{eq:est_boyddunster_unif} follows.

We now turn to the second inequality in \eqref{eq:est_boyddunster}. For this we need a better control of the ``error term'' $E^{-1}_m M_m$ in \eqref{eq:boyddunster_Yapprox}. According to \cite[\S 3]{BoydDunster},
\[
E^{-1}_m M_m(z) = \sqrt{2} J_m(z)
\]
for all $z \in [0,X_m]$, where $X_m$ is a positive real number defined in \cite[eq.\ (3.4)]{BoydDunster} and satisfying $X_m \geq m$ by \cite[Corollary 1 applied with $\theta = 3\pi/4$]{MuSp}. Hence \eqref{eq:boyddunster_Yapprox} yields that
\begin{multline}\label{eq:boyddunster_bound_Y_better}
|x^2-a_{\ell,m}^2|^{1/4} \, |\tSpH_{\ell,-m}(x)| \\
\leq C_\epsilon |(\ell+1/2)^2 \zeta_{\ell,m}(x)-m^2|^{1/4}
|J_{m}((\ell+1/2) \,\zeta_{\ell,m}(x)^{1/2}) |
\end{multline}
uniformly for all $(\ell,m) \in \IS$ with $0 \leq m \leq \epsilon (\ell+1/2)$ and $x \in [0,1]$ satisfying $(\ell+1/2) \, \zeta_{\ell,m}(x)^{1/2} \leq m$, that is, $\zeta_{\ell,m}(x)^{1/2} \leq b_{\ell,m}$.

We now claim that
\begin{equation}\label{eq:boyddunster_claim}
\zeta_{\ell,m}(x)^{1/2} \leq \sqrt{1-x^2}
\end{equation}
for all $x \in [a_{\ell,m},1]$. Assuming the claim, from \eqref{eq:boyddunster_bound_Y_better} and Lemma \ref{lem:besselestimate}, we obtain that, for any given $\delta \in (0,1)$, for all $(\ell,m) \in \IS$ and $x \in [0,1]$ satisfying $0 \leq m \leq \epsilon (\ell+1/2)$ and $\sqrt{1-x^2} \leq \delta b_{\ell,m}$,
\begin{equation}
\begin{split}
|\tSpH_{\ell,-m}(x)| &\leq C_{\epsilon,\delta} \,b_{\ell,m}^{-1/2} m^{1/2} \, |J_{m}((\ell+1/2) \,\zeta_{\ell,m}(x)^{1/2}) | \\
&\leq C_{\epsilon,\delta} \, b_{\ell,m}^{-1/2} m^{1/2} ((\ell+1/2) \,\zeta_{\ell,m}(x)^{1/2}/2)^m / m! \\
&\leq C_{\epsilon,\delta} \, b_{\ell,m}^{-1/2} m^{1/2} \frac{(\ell+1/2)^{m} (1-x^2)^{m/2}}{2^m \, m!} \\
&\leq C_{\epsilon,\delta} \, b_{\ell,m}^{-1/2} m^{1/2} \frac{(\delta m)^m}{2^m \, m!} \\
&\leq C_{\epsilon,\delta} \, b_{\ell,m}^{-1/2} (\delta e/2)^m\,,
\end{split}
\end{equation}
where Stirling's approximation was used in the last step. The second inequality in \eqref{eq:est_boyddunster} follows by choosing $\delta = 1/4$.

We are left with the proof of \eqref{eq:boyddunster_claim}.
Note that the change of variable $\xi = t^2$ gives
\[
\int_{\zeta_{\ell,m}(x)}^{b_{\ell,m}^2} \frac{(b_{\ell,m}^2-\xi)^{1/2}}{2\xi} \,d\xi
 = \int^{b_{\ell,m}}_{\zeta_{\ell,m}(x)^{1/2}} \frac{(b_{\ell,m}^2-t^2)^{1/2}}{t} \,dt,
\]
while the change of variable $s = \sqrt{1-t^2}$ gives
\[
\int_{a_{\ell,m}}^x \frac{(s^2-a_{\ell,m}^2)^{1/2}}{1-s^2} \,ds = \int^{b_{\ell,m}}_{\sqrt{1-x^2}} \frac{(b_{\ell,m}^2-t^2)^{1/2}}{t \sqrt{1-t^2}}  \,dt,
\]
hence, by \eqref{eq:def_BDzeta_lower},
\[
\int^{b_{\ell,m}}_{\zeta_{\ell,m}(x)^{1/2}} \frac{(b_{\ell,m}^2-t^2)^{1/2}}{t} \,dt
= \int^{b_{\ell,m}}_{\sqrt{1-x^2}} \frac{(b_{\ell,m}^2-t^2)^{1/2}}{t \sqrt{1-t^2}}  \,dt
\geq \int^{b_{\ell,m}}_{\sqrt{1-x^2}} \frac{(b_{\ell,m}^2-t^2)^{1/2}}{t}  \,dt,
\]
which implies \eqref{eq:boyddunster_claim}.
\end{proof}

By combining Propositions \ref{prp:est_olver} and \ref{prp:est_boyddunster} we obtain in particular the following estimate.

\begin{corollary}\label{cor:est_combinata}
For all $(\ell,m) \in \IS$ and $x \in [-1,1]$,
\[
|\tSpH_{\ell,m}(x)| \leq C \left( (1+|m|) (1+\ell)^{-2} + |x^2 -a_{\ell,m}^2| \right)^{-1/4}.
\]
\end{corollary}

\section{A weighted Plancherel-type estimate}\label{s:weighted}

As a consequence of \eqref{eq:base}
for any bounded Borel function $G:\RR^2\to \CC$ we have
\begin{equation}\label{eq:mult}
G(\opL,T) f(z)=\sum_{(\ell,m) \in \IS} G(\lambda_{\ell,m},m)\,\langle f, \SpH_{\ell,m}\rangle \SpH_{\ell,m} (z),
\end{equation}
for all $f \in L^2(\sfera)$ and almost all $z\in\sfera$.
The integral kernel $\Kern_{G(\opL,T)}$ of the operator $G(\opL,T) $  is  then given by
\[
\Kern_{G(\opL,T)} (z,z')=  \sum_{(\ell,m)\in\IS} G(\lambda_{\ell,m},m)\, \SpH_{\ell,m}  (z) \overline{\SpH_{\ell,m} (z')},
\]
and satisfies
\begin{equation}\label{eq:nucleoquadrato}
\|\Kern_{G(\opL,T)} (\cdot,\spnt{\theta'}{\varphi'})\|_{L^2(\sfera)}^2=  \sum_{(\ell,m)\in\IS} \left|G(\lambda_{\ell,m},m)\right|^2 \left| \tSpH_{\ell,m}  (\sin\theta')\right|^2.
\end{equation}

We are going to prove a weighted Plancherel-type estimate for $\opL$.
To this purpose, the following elementary lemma will be of use: it gives a sufficient condition for a sum to be estimated by a corresponding integral (cf., e.g., \cite[proof of Lemma 3.4]{CS}).

\begin{lemma}\label{lem:sumintegral}
Let $\kappa \in [1,\infty)$. Let $D \subseteq \RR$ be open and $\phi : D \to \RR$ be a nonnegative differentiable function satisfying
\[
|\phi'(x)| \leq \kappa \phi(x)
\]
for all $x \in D$. Let $R \subseteq \RR$ be such that
\[
\inf \{ |x-x'| \tc x,x'\in R, \, x\neq x'\} \geq \kappa^{-1}.
\]
Then, for all intervals $I \subseteq D$ with length $|I| \geq \kappa^{-1}$,
\[
\sum_{x \in R \cap I} \phi(x) \leq C_\kappa \int_I \phi(x) \,dx,
\]
where the constant $C_\kappa$ depends only on $\kappa$ and not on $I$, $R$, $\phi$.
\end{lemma}
\begin{proof}
Without loss of generality we may assume that $R$ is finite, contained in $I$ and nonempty. Let $x_0$ be the maximum of $R$.

The above differential inequality for $\phi$ implies that
\[
\phi(x') \leq e^{\kappa |x'-x|} \phi(x)
\]
whenever $x,x' \in I$. For all $x \in R \setminus \{x_0\}$, the interval $[x,x+\kappa^{-1})$ is contained in $I$ (indeed $[x,x+\kappa^{-1}) \subseteq [x,x_0] \subseteq I$, since $R$ is $\kappa^{-1}$-separated and $I$ is an interval) and moreover the intervals $[x,x+\kappa^{-1})$ with $x$ ranging in $R \setminus \{x_0\}$ are pairwise disjoint (again because $R$ is $\kappa^{-1}$-separated). Hence
\[
\sum_{x \in R \setminus\{x_0\}} \phi(x) \leq \sum_{x \in R \setminus\{x_0\}} \kappa \int_{[x,x+\kappa^{-1})} e^{\kappa|x-x'|} \phi(x') \,dx' \leq e \kappa \int_I \phi(x) \,dx.
\]
Similarly, since $|I|\geq \kappa^{-1}$, there exists an interval $J \subseteq I$ containing $x_0$ with length $|J|=\kappa^{-1}$ and therefore
\[
\phi(x_0) \leq |J|^{-1} \int_J e^{\kappa |x_0-x|} \phi(x) \,dx \leq e\kappa \int_I \phi(x) \,dx.
\]
Hence the conclusion follows with $C_\kappa = 2e\kappa$.
\end{proof}

Set $[\ell] = \ell+1/2$ for all $\ell \in \NN$.

\begin{proposition}\label{prp:indici_alti}
Let $\epsilon \in (0,1)$. For all $i \in \NN \setminus \{0\}$ and $\alpha \in [0,1/2)$,
\begin{equation}\label{eq:indici_alti_orig}
\sup_{x \in [-1,1]}
\frac{1}{i} \max\left\{ \frac{1}{i}, |x|  \right\}^{1-2\alpha} \sum_{\substack{(\ell,m)\in\IS \\ |m| \geq \epsilon [\ell] \\ \lambda_{\ell,m} \in [i^2,(i+1)^2]}} \lambda_{\ell,m}^{\alpha} \, |m|^{-2\alpha} \left| \tSpH_{\ell,m}(x) \right|^2
\leq
C_{\epsilon,\alpha}.
\end{equation}
\end{proposition}
\begin{proof}
Note that, by \eqref{eq:Yindexparity}, it is enough to prove the above estimate with the sum restricted to $m \geq 0$. Since $|m| \simeq [\ell]$ and $\lambda_{\ell,m} \simeq i^2$ in the summation range, we are reduced to proving the following estimate: for all $x \in [-1,1]$,
\begin{equation}\label{eq:indici_alti_mod}
\sum_{\substack{(\ell,m)\in\IS \\ |m| \geq \epsilon [\ell] \\ \lambda_{\ell,m} +1/4 \in [i^2,(i+1)^2]}} i^{2\alpha} \, [\ell]^{-2\alpha} \left| \tSpH_{\ell,m}(x) \right|^2 \leq C_{\epsilon,\alpha} \, i \max\left\{ \frac{1}{i}, |x|  \right\}^{2\alpha-1} .
\end{equation}
Indeed, $\lambda_{\ell,m} \in [i^2,(i+1)^2]$ implies $\lambda_{\ell,m} +1/4 \in [i^2,(i+2)^2]$, so \eqref{eq:indici_alti_orig} follows by combining two instances of \eqref{eq:indici_alti_mod} (corresponding to $i$ and $i+1$ respectively).

It is convenient to reindex the sum by setting $p = \ell+m+1/2$ and $q = \ell-m+1/2$, so
\[
\lambda_{\ell,m} + 1/4 = p q , \qquad a_{\ell,m}^2 = \frac{4pq}{(p+q)^2};
\]
moreover the range $(\ell,m) \in \IS$ corresponds to $(p,q) \in (\NN+1/2)^2$, the condition $m \geq 0$ corresponds to $p \geq q$, and $m \geq \epsilon [\ell]$ corresponds to $\bar \epsilon p \geq q$, where $\bar \epsilon = (1-\epsilon)/(1+\epsilon) \in (0,1)$.

Let us first discuss the range $|x| \leq a_{\ell,m}/2$.
In this range, by Proposition \ref{prp:est_olver},
\[
|\tSpH_{\ell,m}(x)|^2 \lesssim |x^2-a_{\ell,m}^2|^{-1/2} \lesssim a_{\ell,m}^{-1} \simeq p/i \simeq i/q
\]
and moreover $|x| \lesssim a_{\ell,m} \simeq q/i$, i.e., $i \gtrsim q \gtrsim i|x|$ (since $a_{\ell,m} \leq 1$).
Hence
\[\begin{split}
\sum_{\substack{(\ell,m)\in\IS \\ |m| \geq \epsilon [\ell] \\ \lambda_{\ell,m} +1/4 \in [i^2,(i+1)^2] \\ |x| \leq a_{\ell,m}/2}} i^{2\alpha} \, [\ell]^{-2\alpha} \left| \tSpH_{\ell,m}(x) \right|^2
&\lesssim i^{2\alpha} \sum_{\substack{q\in\NN+1/2 \\ i \gtrsim q \gtrsim i |x|}} \sum_{\substack{p\in\NN+1/2\\ p \in [i^2/q,(i+1)^2/q]}} \frac{i}{q} \, p^{-2\alpha} \\
&\lesssim i^{2\alpha} \sum_{\substack{q\in\NN+1/2 \\ q \gtrsim i |x|}} \frac{i^2}{q^2} \, i^{-4\alpha} q^{2\alpha} \\
&\lesssim i^{2-2\alpha} \max\{i|x|,1\}^{2\alpha-1} \\
&= i \max\{|x|,i^{-1}\}^{2\alpha-1},
\end{split}\]
where we used the fact the interval $[i^2/q,(i+1)^2/q]$ has length $(2i+1)/q \simeq i/q \gtrsim 1$.

Let us now consider the range $|x| \geq K a_{\ell,m}$, where $K$ is the constant given by Proposition \ref{prp:est_olver}. In this range
\[
|\tSpH_{\ell,m}(x)|^2 \lesssim |x|^{-1} \exp(-c p |x|^2) \lesssim p^{1/2} (p |x|^2)^{-N}
\]
for some $c \in (0,1)$, where $N$ is an arbitrarily large exponent (to be fixed later). Moreover $|x| \gtrsim q/i$, i.e., $q \lesssim i|x|$, and in particular $i|x| \gtrsim 1$. Hence
\[\begin{split}
\sum_{\substack{(\ell,m)\in\IS \\ |m| \geq \epsilon [\ell] \\ \lambda_{\ell,m} +1/4 \in [i^2,(i+1)^2] \\ |x| \geq K a_{\ell,m}}} &i^{2\alpha} \, [\ell]^{-2\alpha} \left| \tSpH_{\ell,m}(x) \right|^2 \\
&\lesssim i^{2\alpha} \sum_{\substack{q\in\NN+1/2 \\ q\lesssim i |x|}} \sum_{\substack{p\in\NN+1/2\\ p \in [i^2/q,(i+1)^2/q]}} p^{1/2-2\alpha-N} |x|^{-2N} \\
&\lesssim i^{2\alpha} |x|^{-2N} \sum_{\substack{q\in\NN+1/2 \\ q\lesssim i|x|}} \frac{i}{q} \left(\frac{i^2}{q}\right)^{1/2-2\alpha-N} \\
&\lesssim i^{2-2\alpha-2N} |x|^{-2N} (i|x|)^{N+2\alpha-1/2}  \\
&= (i|x|)^{1/2-N} \, i |x|^{2\alpha-1} \lesssim i |x|^{2\alpha-1} \lesssim i \max\{|x|,i^{-1}\}^{2\alpha-1},
\end{split}\]
as long as we choose $N > 1/2$. Again, the fact that the interval $[i^2/q,(i+1)^2/q]$ has length $\simeq i/q \gtrsim 1$ was used.

We are left with the range $a_{\ell,m}/2 \leq |x| \leq K a_{\ell,m}$. In this range,
\[
|x| \simeq a_{\ell,m} \simeq i/p \gtrsim 1/i.
\]
Note that $a_{\ell,m}^2 = \phi(q/p)$, where $\phi(w) = 4w/(1+w)^2$. Note moreover that $\phi : [0,1] \to [0,1]$ is an increasing bijection, satisfying $w \leq \phi(w) \leq 4w$; its derivative is given by
$\phi'(w) = 4 \frac{1-w}{(1+w)^3}$
and vanishes only at $w=1$.
Hence, if we set $\bar x = \sqrt{\phi^{-1}(x^2)}$, then
\[
\bar x \simeq |x|
 \qquad\text{and}\qquad
|x^2 - a_{\ell,m}^2| \simeq |\bar x^2 - q/p|
\]
uniformly for $x \in [0,1]$ and $p,q \in \NN+1/2$ with $q \leq \bar\epsilon p$. Thus, by Proposition \ref{prp:est_olver},
\[
|\tSpH_{\ell,m}(x)|^2 \lesssim (p^{-1} + |\bar x^2-q/p|)^{-1/2} =: \Phi(\bar x,p,q).
\]
The above considerations give that
\[
\sum_{\substack{(\ell,m)\in\IS \\ |m| \geq \epsilon [\ell] \\ \lambda_{\ell,m} +1/4 \in [i^2,(i+1)^2] \\ a_{\ell,m}/2 \leq |x| \leq K a_{\ell,m}}} i^{2\alpha} \, [\ell]^{-2\alpha} \left| \tSpH_{\ell,m}(x) \right|^2
\lesssim |x|^{2\alpha} \sum_{\substack{q\in\NN+1/2 \\ q \simeq i|x|}} \sum_{\substack{p\in\NN+1/2\\ p \in [i^2/q,(i+1)^2/q]}} \Phi(\bar x,p,q) .
\]

We now split the sum in $q$ into three parts. Let us consider first the part where $q < i\bar x$.
Note that
\[
\frac{\partial \Phi}{\partial p}(\bar x,p,q) = \frac{1}{2} \Phi(\bar x,p,q) \frac{p^{-2} + qp^{-2} \sgn (q/p - \bar x^2)}{p^{-1} + |\bar x^2 - q/p|},
\]
so
\begin{equation}\label{eq:PhiLipschitz}
\left|\frac{\partial \Phi}{\partial p}(\bar x,p,q)\right| \lesssim \Phi(\bar x,p,q)
\end{equation}
whenever $p,q \geq 1/2$ and $q \lesssim p$.
Note moreover that the interval $[i^2/q,(i+1)^2/q]$ has length $(2i+1)/q \gtrsim 1$ whenever $q \lesssim i$. By Lemma \ref{lem:sumintegral},
we can then estimate the sum in $p$ by the corresponding integral:
\[\begin{split}
& |x|^{2\alpha} \sum_{\substack{q\in\NN+1/2 \\ q \simeq i|x|\\q < i\bar x}} \sum_{\substack{p\in\NN+1/2\\ p \in [i^2/q,(i+1)^2/q]}} \Phi(\bar x,p,q) \\
&\lesssim |x|^{2\alpha} \sum_{\substack{q\in\NN+1/2 \\ q \simeq i|x|\\q < i\bar x}} \int_{p \in [i^2/q,(i+1)^2/q]} \Phi(\bar x,p,q) \,dp \\
&\leq |x|^{2\alpha} \sum_{\substack{q\in\NN+1/2 \\ q \simeq i|x|\\q < i\bar x}} \int_{p \in [i^2/q,(i+1)^2/q]} |\bar x^2 - q/p|^{-1/2} \,dp \\
&\lesssim i|x|^{2\alpha-2} \sum_{\substack{q\in\NN+1/2 \\ q \simeq i|x|\\q < i\bar x}} \int_{v \in [\bar x^2 i^2/q^2,\bar x^2 (i+1)^2/q^2]} |v-1|^{-1/2} \,dv ,
\end{split}\]
where the change of variable $v=\bar x^2 p/q$ was made in the last step, and the fact that $v \simeq \bar x^2 i^2/q^2 \simeq 1$ was used. Thanks to the condition $q < i\bar x$, we have $v > 1$ in the domain of integration, so
\[\begin{split}
\int_{v \in [\bar x^2 i^2/q^2,\bar x^2 (i+1)^2/q^2]} |v-1|^{-1/2} \,dv &= 2[ (\bar x^2 (i+1)^2/q^2-1)^{1/2} - (\bar x^2 i^2/q^2-1)^{1/2} ] \\
&\lesssim  i^{-1} (1-q^2/(\bar x^2 i^2))^{-1/2},
\end{split}\]
where the fact that $q \simeq i\bar x$ was used. Hence
\[\begin{split}
& |x|^{2\alpha} \sum_{\substack{q\in\NN+1/2 \\ q \simeq i|x|\\q < i\bar x}} \sum_{\substack{p\in\NN+1/2\\ p \in [i^2/q,(i+1)^2/q]}} \Phi(\bar x,p,q) \\
&\lesssim |x|^{2\alpha-2} \sum_{\substack{q\in\NN+1/2 \\ q \simeq i|x|\\q < i\bar x}} (1-q^2/(\bar x^2 i^2))^{-1/2}\\
&\lesssim i|x|^{2\alpha-1} \int_{\substack{q \simeq i|x|\\q < i\bar x}} (1-q^2/(\bar x^2 i^2))^{-1/2} \frac{dq}{i \bar x}\\
&\lesssim i|x|^{2\alpha-1} \lesssim i \max\{i^{-1},|x|\}^{2\alpha-1},
\end{split}\]
since $|x| \gtrsim 1/i$ in this range.

Similarly one can control the part of the sum where $q > (i+1)\bar x$. Finally, the part of the sum where $i \bar x \leq q \leq (i+1)\bar x$ contains at most two summands and $\Phi(\bar x,p,q) \leq p^{1/2}$, so
\[\begin{split}
|x|^{2\alpha} \sum_{\substack{q\in\NN+1/2 \\ i\bar x \leq q \leq (i+1)\bar x}} \sum_{\substack{p\in\NN+1/2\\ p \in [i^2/q,(i+1)^2/q]}} \Phi(\bar x,p,q)
&\lesssim |x|^{2\alpha} \sum_{\substack{q\in\NN+1/2 \\ i\bar x \leq q \leq (i+1)\bar x}} \frac{i}{q} \left(\frac{i^2}{q}\right)^{1/2} \\
&\lesssim \frac{i|x|^{2\alpha-1}}{\sqrt{i|x|}} \lesssim i \max\{i^{-1},|x|\}^{2\alpha-1},
\end{split}\]
since $|x| \gtrsim 1/i$ in this range.
\end{proof}

\begin{proposition}\label{prp:indici_bassi}
Let $\epsilon \in (0,1)$. For all $i \in \NN \setminus \{0\}$,
\begin{equation}
\sup_{x \in [-1,1]}
\frac{1}{i} \sum_{\substack{(\ell,m)\in\IS \\ |m| \leq \epsilon [\ell] \\ \lambda_{\ell,m} \in [i^2,(i+1)^2]}} \left| \tSpH_{\ell,m}(x) \right|^2 \leq C_\epsilon
\end{equation}
\end{proposition}
\begin{proof}
Similarly as in the proof of Proposition \ref{prp:indici_alti}, it is immediately seen that it suffices to prove the following estimate: for all $x \in [-1,1]$,
\begin{equation}
\sum_{\substack{(\ell,m)\in\IS \\ 0 \leq m \leq \epsilon [\ell] \\ \lambda_{\ell,m}+1/4 \in [i^2,(i+1)^2]}} \left| \tSpH_{\ell,m}(x) \right|^2 \leq C_\epsilon \, i.
\end{equation}
Note that
\[
\lambda_{\ell,m} +1/4 = [\ell]^2 - m^2.
\]

Let us consider first the part of the sum where $m=0$. Here the condition $\lambda_{\ell,0} + 1/4 \in [i^2,(i+1)^2]$ uniquely determines the value of $\ell$ and moreover $|\tSpH_{\ell,0}(x)|^2 \lesssim [\ell] \simeq i$ by Proposition \ref{prp:stime-classiche}\ref{en:stime_unif}, so
\[
\sum_{\substack{(\ell,m)\in\IS \\ m=0 \\ \lambda_{\ell,m}+1/4 \in [i^2,(i+1)^2]}} \left| \tSpH_{\ell,m}(x) \right|^2 \lesssim i.
\]
Hence in the rest of the proof we may assume $m > 0$.

Let us introduce the notation $y = \sqrt{1-x^2}$.
Recall also the definition of $b_{\ell,m}$ in
\eqref{eq:criticalpoints}
and observe that the condition $m \leq \epsilon [\ell]$ corresponds to $b_{\ell,m} \leq \epsilon$.

In the part of the sum where
$y \geq \epsilon^{-1/2}  b_{\ell,m}$, we have $|\tSpH_{\ell,m}(x)|^2 \lesssim y^{-1}$ by Proposition \ref{prp:est_boyddunster}, and moreover $y \gtrsim m/[\ell] \simeq m/i$. Hence
\[
\sum_{\substack{\ell,m \in \NN \\ 0< m \leq \epsilon [\ell] \\ [\ell]^2-m^2 \in [i^2,(i+1)^2] \\ \sqrt{1-x^2} \geq \epsilon^{-1/2} b_{\ell,m}}} \left| \tSpH_{\ell,m}(x) \right|^2 \lesssim \sum_{\substack{m \in \NN+1 \\ m \lesssim iy}} \sum_{\substack{\ell \in \NN \\ [\ell] \in [\sqrt{m^2+i^2},\sqrt{m^2+(i+1)^2}]}} y^{-1} \lesssim i,
\]
where we used the fact that the length of the interval $[\sqrt{m^2+i^2},\sqrt{m^2+(i+1)^2}]$ is controlled by $i/\sqrt{m^2+i^2} \simeq 1$.

In the part of the sum where $y \leq b_{\ell,m}/4$, by  Proposition \ref{prp:est_boyddunster} we have $|\tSpH_{\ell,m}(x)|^2 \lesssim b_{\ell,m}^{-1} 2^{-2m} \lesssim i 2^{-2m}$, and therefore
\[
\sum_{\substack{\ell,m \in \NN \\ m \leq \tilde\epsilon [\ell] \\ [\ell]^2-m^2 \in [i^2,(i+1)^2] \\ \sqrt{1-x^2} \leq b_{\ell,m}/4}} \left| \tSpH_{\ell,m}(x) \right|^2
\lesssim \sum_{m  \in \NN +1} \sum_{\substack{\ell \in \NN \\ [\ell] \in [\sqrt{m^2+i^2},\sqrt{m^2+(i+1)^2}]}} i 2^{-2m} \lesssim i,
\]
where again we used the fact that the length of $[\sqrt{m^2+i^2},\sqrt{m^2+(i+1)^2}]$ is controlled by $1$.

Finally we consider the part where $b_{\ell,m}/4 \leq y \leq \epsilon^{-1/2}  b_{\ell,m}$. In this range, $y \simeq m/[\ell] \simeq m/i$, and actually it must be
\[
y \leq \epsilon^{-1/2} \epsilon = \epsilon^{1/2} < 1;
\]
so, if we set $\bar y = y(1-y^2)^{-1/2}$, then
\[
y \simeq \bar y
\]
uniformly.
Moreover, by Corollary \ref{cor:est_combinata},
\[
|\tSpH_{\ell,m}(x)|^2 \lesssim \Psi(y,[\ell],m),
\]
where
\[
\Psi(y,\vell,m) = \left( \frac{m}{\vell^2} + \left|y^2-\frac{m^2}{\vell^2}\right|\right)^{-1/2}.
\]
Hence
\[
\sum_{\substack{\ell,m \in \NN \\ 0 < m \leq \epsilon [\ell] \\ [\ell]^2-m^2 \in [i^2,(i+1)^2] \\ b_{\ell,m}/4 \leq \sqrt{1-x^2} \leq \epsilon^{-1/2} b_{\ell,m}}} \left| \tSpH_{\ell,m}(x) \right|^2
\lesssim \sum_{\substack{m \in \NN +1 \\ m \simeq iy}} \sum_{\substack{\vell \in \NN+1/2 \\ \vell \in [\sqrt{m^2+i^2},\sqrt{m^2+(i+1)^2}]}} \Psi(y,\vell,m) \\
\]

We now split the sum in $m$ into three parts. Consider first the part where $m < i\bar y$.
Note that
\[
\frac{\partial \Psi}{\partial \vell}(y,\vell,m) = \Psi(y,\vell,m) \frac{\frac{m}{\vell^3} + \frac{m^2}{\vell^3} \sgn(\frac{m^2}{\vell^2}-y^2)}{\frac{m}{\vell^2} + \left|y^2-\frac{m^2}{\vell^2}\right|},
\]
so
\begin{equation}\label{eq:PsiLipschitz}
\left|\frac{\partial \Psi}{\partial \vell}(y,\vell,m)\right| \lesssim \Psi(y,\vell,m)
\end{equation}
whenever $1 \leq m \leq \vell$.
Moreover the length of $[\sqrt{m^2+i^2},\sqrt{m^2+(i+1)^2}]$ is bounded above and below uniformly by constants, since $m \lesssim i$. By Lemma \ref{lem:sumintegral}, we can then estimate the sum in $\vell$ by the corresponding integral:
\[\begin{split}
&\sum_{\substack{m \in \NN +1 \\ m \simeq iy \\ m < i \bar y}} \sum_{\substack{\vell \in \NN+1/2 \\ \vell \in [\sqrt{m^2+i^2},\sqrt{m^2+(i+1)^2}]}} \Psi(y,\vell,m) \\
&\lesssim \sum_{\substack{m \in \NN \\ m \simeq iy \\ m < i \bar y}} \int_{\substack{\vell \in [\sqrt{m^2+i^2},\sqrt{m^2+(i+1)^2}]}} \Psi(y,\vell,m) \,d\vell .
\end{split}\]
Note further that
\[
\begin{split}
\Psi(y,\vell,m)
&\leq \left|y^2-\frac{m^2}{\vell^2}\right|^{-1/2}  \\
&= (1-y^2)^{-1/2} \frac{\vell}{m} \left|\bar y^2 \frac{\vell^2-m^2}{m^2}-1\right|^{-1/2}
\lesssim \frac{i}{\bar y} \frac{\bar y^2 \vell}{m^2} \left|\bar y^2 \frac{\vell^2-m^2}{m^2}-1\right|^{-1/2}
\end{split}
\]
because $m \simeq i \bar y$.
The change of variables $u = \bar y^2 (\vell^2 -m^2)/m^2$ then gives
\[\begin{split}
&\sum_{\substack{m \in \NN+1 \\ m \simeq iy \\ m < i \bar y}} \sum_{\substack{\vell \in \NN +1/2 \\ \vell \in [\sqrt{m^2+i^2},\sqrt{m^2+(i+1)^2}]}} \Psi(y,\vell,m) \\
&\lesssim \frac{i}{\bar y} \sum_{\substack{m \in \NN +1\\ m \simeq iy \\ m < i\bar y}} \int_{\substack{u \in [\bar y^2 i^2/m^2,\bar y^2 (i+1)^2/m^2]}} \left|u-1\right|^{-1/2} \,du .
\end{split}\]
The condition $m < i\bar y$ implies that $u>1$ in the domain of integration, hence
\[\begin{split}
&\int_{\substack{u \in [\bar y^2 i^2/m^2,\bar y^2 (i+1)^2/m^2]}} \left|u-1\right|^{-1/2} \,du \\
&= 2 [(\bar y^2 (i+1)^2/m^2 -1)^{1/2} - (\bar y^2 i^2/m^2 -1)^{1/2} ] \\
&\lesssim i^{-1} (1-m^2/(\bar y^2 i^2))^{-1/2}
\end{split}\]
where the fact that $m \simeq i\bar y$ was used. Therefore
\[\begin{split}
&\sum_{\substack{m \in \NN +1 \\ m \simeq iy \\ m < i \bar y}} \sum_{\substack{\vell \in \NN +1/2 \\ \vell \in [\sqrt{m^2+i^2},\sqrt{m^2+(i+1)^2}]}} \Psi(y,\vell,m) \\
&\lesssim \bar y^{-1} \sum_{\substack{m \in \NN+1 \\ m \simeq iy \\ m < i\bar y}} (1-m^2/(\bar y^2 i^2))^{-1/2}  \\
&\lesssim i \int_{\substack{m \simeq iy \\ m < i\bar y}} (1-m^2/(\bar y^2 i^2))^{-1/2} \,\frac{dm}{i\bar y}
\lesssim i .
\end{split}\]

A similar estimate can be obtained in the part of the sum where $m > (i+1) \bar y$. Finally, in the part where $i \bar y \leq m \leq (i+1) \bar y$, the number of summands in $m$ is bounded by a constant, hence we obtain the bound
\[\begin{split}
&\sum_{\substack{m \in \NN +1 \\ i \bar y \leq m \leq (i+1) \bar y}} \sum_{\substack{\vell \in \NN +1/2 \\ \vell \in [\sqrt{m^2+i^2},\sqrt{m^2+(i+1)^2}]}} \Psi(y,\vell,m)  \\
&\lesssim \sum_{\substack{m \in \NN+1 \\ i \bar y \leq m \leq (i+1) \bar y}} \sum_{\substack{\vell \in \NN +1/2 \\ \vell \in [\sqrt{m^2+i^2},\sqrt{m^2+(i+1)^2}]}} \vell/m^{1/2} \\
&\lesssim \frac{i}{\sqrt{i \bar y}} \lesssim i
\end{split}\]
where we used the fact that
$i\bar y \simeq m \gtrsim 1$.
\end{proof}

The previous estimates allow us to prove a ``weighted Plancherel-type estimate'' for the Grushin operator $\opL$.  For all $r \in (0,\infty)$, define the weight $\weight_r : \sfera \times \sfera \to [0,\infty)$ by
\begin{equation}\label{eq:weight}
\weight_r(\spnt{\theta}{\varphi},\spnt{\theta'}{\varphi'}) = \frac{|\theta|}{\max\{r,|\theta'|\}}.
\end{equation}

Similarly as in \cite{CoS,DOS}, for all $N \in \NN \setminus \{0\}$ and $F : \RR \to \CC$ supported in $[0,1]$, let the norm $\|F\|_{N,2}$ be defined by
\[
\|F\|_{N,2} = \left(\frac{1}{N} \sum_{i=1}^N \sup_{\lambda \in [(i-1)/N,i/N]} |F(\lambda)|^2\right)^{1/2}.
\]
Moreover, similarly as in \cite{M}, for all $r \in (0,\infty)$, $\alpha,\beta \in [0,\infty)$, $p \in [1,\infty]$ and $K : \sfera \times \sfera \to \CC$, let the norm $\vvvert K \vvvert_{p,\beta,\alpha,r}$ be defined by
\[
\vvvert K \vvvert_{p,\beta,\alpha,r}
= \esssup_{z' \in \sfera} V(z',r)^{1/p'} \| (1+\dist(\cdot,z')/r)^\beta \, (1+\weight_r(\cdot,z'))^\alpha \, K(\cdot,z') \|_{L^p(\sfera)},
\]
where $p'=p/(p-1)$ is the conjugate exponent to $p$.

\begin{proposition}\label{prp:weighted_plancherel}
Let $\alpha \in [0,1/2)$ and $N \in \NN \setminus \{0\}$. For all Borel functions $F : \RR \to \CC$ supported in $[0,N]$,
\[
\vvvert \Kern_{F(\sqrt{\opL})} \vvvert_{2,0,\alpha,N^{-1}}
\leq C_\alpha \|F(N \cdot)\|_{N,2}
\]
\end{proposition}
\begin{proof}
Note that it is enough to prove that
\begin{equation}\label{eq:weighted_reduced}
\| \weight_{N^{-1}}(\cdot,z')^\alpha \, \Kern_{F(\sqrt{\opL})}(\cdot,z') \|_{L^2(\sfera)} \leq C_\alpha V(z',N^{-1})^{-1/2} \|F(N \cdot)\|_{N,2}
\end{equation}
for all $z' \in \sfera$.
Indeed Proposition \ref{prp:weighted_plancherel} follows by combining the estimate \eqref{eq:weighted_reduced} with the analogous one where $\alpha = 0$.

We now decompose
\[\begin{split}
\Kern_{F(\sqrt{\opL})}(z,z')
&= \sum_{(\ell,m) \in \IS} F(\sqrt{\lambda_{\ell,m}}) \,  \SpH_{\ell,m}  (z) \overline{\SpH_{\ell,m} (z')} \\
&= \sum_{\substack{(\ell,m) \in \IS \\ |m| \leq [\ell]/2}} + \sum_{\substack{(\ell,m) \in \IS \\ |m| > [\ell]/2}} =: K_1(z,z') +K_2(z,z'),
\end{split}\]
and observe that
\[
K_2(\cdot,z') = \opL^{-\alpha/2} |T|^\alpha K_{2,\alpha} (\cdot,z'),
\]
where
\[
K_{2,\alpha}(z,z') = \sum_{\substack{(\ell,m) \in \IS \\ |m| > [\ell]/2}} \lambda_{\ell,m}^{\alpha/2} \, |m|^{-\alpha} \, F(\sqrt{\lambda_{\ell,m}}) \,  \SpH_{\ell,m}(z) \overline{\SpH_{\ell,m}(z')}.
\]

Hence
\[\begin{split}
&\| \weight_{N^{-1}}(\cdot,\spnt{\theta'}{\varphi'})^\alpha \Kern_{F(\sqrt{\opL})}(\cdot,\spnt{\theta'}{\varphi'}) \|_{L^2(\sfera)} \\
&\leq
\| \weight_{N^{-1}}(\cdot,\spnt{\theta'}{\varphi'})^\alpha K_1(\cdot,\spnt{\theta'}{\varphi'}) \|_{L^2(\sfera)} +
\| \weight_{N^{-1}}(\cdot,\spnt{\theta'}{\varphi'})^\alpha K_2(\cdot,\spnt{\theta'}{\varphi'}) \|_{L^2(\sfera)} \\
&\lesssim
\min\{N,|\theta'|^{-1}\}^\alpha \left[ \| K_1(\cdot,\spnt{\theta'}{\varphi'}) \|_{L^2(\sfera)} +
\| \mathrm{t}^\alpha K_2(\cdot,\spnt{\theta'}{\varphi'}) \|_{L^2(\sfera)} \right] \\
&\leq
\min\{N,|\theta'|^{-1}\}^\alpha \left[ \| K_1(\cdot,\spnt{\theta'}{\varphi'}) \|_{L^2(\sfera)} +
\| K_{2,\alpha}(\cdot,\spnt{\theta'}{\varphi'}) \|_{L^2(\sfera)} \right]
\end{split}\]
where $\mathrm{t}(\spnt{\theta}{\varphi}) = |\tan\theta|$ and in the last step Lemma \ref{lem:weighted_est} was used.
By \eqref{eq:volume}, the desired estimate \eqref{eq:weighted_reduced} is then reduced to proving that
\begin{align}
\| K_1(\cdot,\spnt{\theta'}{\varphi'}) \|_{L^2(\sfera)}^2 &\leq C_\alpha N^2 \min\{ N, |\theta'|^{-1} \}^{1-2\alpha} \|F(N\cdot)\|_{N,2}^2, \label{eq:est_bassi_orig}\\
\| K_{2,\alpha}(\cdot,\spnt{\theta'}{\varphi'}) \|_{L^2(\sfera)}^2 &\leq C_\alpha N^2 \min\{ N, |\theta'|^{-1} \}^{1-2\alpha} \|F(N\cdot)\|_{N,2}^2 \label{eq:est_alti},
\end{align}
and actually, instead of \eqref{eq:est_bassi_orig}, we shall prove the stronger estimate
\begin{equation}\label{eq:est_bassi}
\| K_1(\cdot,\spnt{\theta'}{\varphi'}) \|_{L^2(\sfera)}^2 \leq C N^2 \|F(N\cdot)\|_{N,2}^2.
\end{equation}

By the orthonormality of the spherical harmonics $\SpH_{\ell,m}$, the estimates \eqref{eq:est_alti} and \eqref{eq:est_bassi} can be rewritten as
\begin{multline*}
\sum_{\substack{(\ell,m) \in \IS \\ |m| > [\ell]/2}} \lambda_{\ell,m}^{\alpha} |m|^{-2\alpha} |F(\sqrt{\lambda_{\ell,m}})|^2 |\SpH_{\ell,m}(\spnt{\theta'}{\varphi'})|^2 \\
\leq C_\alpha N \min\{ N, |\theta'|^{-1} \}^{1-2\alpha} \sum_{i=1}^N \sup_{\lambda \in [i-1,i]} |F(\lambda)|^2
\end{multline*}
and
\[
\sum_{\substack{(\ell,m) \in \IS \\ |m| \leq [\ell]/2}} |F(\sqrt{\lambda_{\ell,m}})|^2 |\SpH_{\ell,m}(\spnt{\theta'}{\varphi'})|^2 \leq C N \sum_{i=1}^N \sup_{\lambda \in [i-1,i]} |F(\lambda)|^2.
\]
So we are reduced to proving that
\[
\sum_{\substack{(\ell,m) \in \IS \\ |m| > [\ell]/2 \\ \lambda_{\ell,m} \in [(i-1)^2,i^2]}} \lambda_{\ell,m}^{\alpha} \, |m|^{-2\alpha} \, |\SpH_{\ell,m}(\spnt{\theta'}{\varphi'})|^2 \leq C_\alpha N \min\{ N, |\theta'|^{-1} \}^{1-2\alpha}
\]
and
\[
\sum_{\substack{(\ell,m) \in \IS \\ |m| \leq [\ell]/2 \\ \lambda_{\ell,m} \in [(i-1)^2,i^2]}} |\SpH_{\ell,m}(\spnt{\theta'}{\varphi'})|^2 \leq C N
\]
for $i=1,\dots,N$. For $i=1$, the above estimates are immediately verified (the sums have at most one summand corresponding to $(\ell,m) = (0,0)$ or $(\ell,m) = (1,1)$ and the spherical harmonics $\SpH_{\ell,m}$ are bounded functions). For $i=2,\dots,N$, these estimates follow from Propositions \ref{prp:indici_alti} and \ref{prp:indici_bassi}.
\end{proof}

As an immediate consequence of the weighted Plancherel-type estimate, we obtain the following on-diagonal bound for the heat propagator associated to $\opL$.

\begin{corollary}\label{cor:heatkernel}
For all $r \in (0,\infty)$,
\[
\vvvert \Kern_{\exp(-r^2 \opL)} \vvvert_{2,0,0,r} \leq C.
\]
\end{corollary}
\begin{proof}
Let $k_0 = \min \{k \in \NN \tc r \geq 2^{-k}\}$.
Let $F_0(\lambda) = \exp(-r^2\lambda^2) \chi_{[0,2^{k_0})}(\lambda)$ and $F_k(\lambda) = \exp(-r^2\lambda^2) \, \chi_{[2^{k_0+k-1},2^{k_0+k})}(\lambda)$ for $k > 0$.
Note that $\supp F_k \subseteq [0,2^{k_0+k}]$. Hence, by Proposition \ref{prp:weighted_plancherel}, for all $k>0$,
\[
\vvvert \Kern_{F_k(\sqrt{\opL})} \vvvert_{2,0,0,2^{-(k_0+k)}} \leq C \| F_k \|_\infty = C \exp(-2^{2(k_0+k-1)} r^2).
\]
So, by the doubling condition,
\[\begin{split}
\vvvert \Kern_{\exp(-t^2 \opL)} \vvvert_{2,0,0,r}
&\leq \sum_{k \in \NN} \vvvert \Kern_{F_k(\sqrt{\opL})} \vvvert_{2,0,0,r} \\
&\leq \vvvert \Kern_{F_0(\sqrt{\opL})} \vvvert_{2,0,0,r} + C \sum_{k > 0} (r 2^{k_0+k})^{Q/2} \vvvert \Kern_{F_k(\sqrt{\opL})} \vvvert_{2,0,0,2^{-(k_0+k)}} \\
&\leq \vvvert \Kern_{F_0(\sqrt{\opL})} \vvvert_{2,0,0,r} + C \sum_{k >0} (r 2^{k_0+k})^{Q/2} \exp(-2^{2(k_0+k-1)} r^2),
\end{split}\]
where $Q = 3$ is the homogeneous dimension of $\sfera$ with the given sub-Riemannian structure. Since $r 2^{k_0} \geq 1$, it is easily seen that the last sum in $k$ is bounded uniformly in $r$, so it remains to control the term  $\vvvert \Kern_{F_0(\sqrt{\opL})} \vvvert_{2,0,0,r}$.

If $k_0 > 0$, then similarly as before it is seen that
\[
\vvvert \Kern_{F_0(\sqrt{\opL})} \vvvert_{2,0,0,r} \leq C (r2^{k_0})^{Q/2} \vvvert \Kern_{F_0(\sqrt{\opL})} \vvvert_{2,0,0,2^{-k_0}} \leq C (r2^{k_0})^{Q/2} \leq C 2^{Q/2},
\]
since $r 2^{k_0-1} \leq 1$. If $k_0 = 0$, then the only eigenvalue of $\sqrt{\opL}$ in $[0,2^{k_0})$ is $0$, hence
\[
\Kern_{F_0(\sqrt{\opL})}(z,z') = \SpH_{0,0}(z) \overline{\SpH_{0,0}(z')} = 1
\]
and trivially $\vvvert \Kern_{F_0(\sqrt{\opL})} \vvvert_{2,0,0,r} \leq \meas(\sfera)^{1/2} \sup_{z' \in \sfera} \| \Kern_{F_0(\sqrt{\opL})}(\cdot,z')\|_{L^2(\sfera)} \leq C$.
\end{proof}

\section{The multiplier theorem}\label{s:multiplier}

We shall need some properties of the weight $\weight_r$ defined in \eqref{eq:weight}. We refer to \cite[Lemma 12]{MSi} and \cite[Lemma 4.1]{M} for similar results.

\begin{lemma}\label{lem:weight}
For all $r>0$ and $\alpha,\beta \geq 0$ such that $\alpha+\beta > 3$ and $\alpha < 1$, and for all $z' \in \sfera$,
\begin{equation}\label{eq:weightsint}
\int_{\sfera} (1+\dist(z,z')/r)^{-\beta} (1+\weight_r(z,z'))^{-\alpha} \,d\meas(z) \leq C_{\alpha,\beta} V(z',r).
\end{equation}
Moreover
\begin{equation}\label{eq:weightsineq}
1+ \weight_r(z,z') \leq C (1+\dist(z,z')/r)
\end{equation}
for all $r>0$ and $z,z' \in \sfera$.
\end{lemma}
\begin{proof}

Due to the compactness of $\sfera$, both
\eqref{eq:weightsint} and
 \eqref{eq:weightsineq}
are trivial for $r\ge 1$.
Thus we  assume from now on that $r<1$.

First, we observe that, as a consequence of \eqref{eq:rho-dist},
\[
\frac{|\theta|}{\max\{r,|\theta'|\}}
\leq 1 + \frac{|\theta-\theta'|}{\max\{r,|\theta'|\}}
\leq 1 + \dist(\spnt{\theta}{\varphi},\spnt{\theta'}{\varphi'})/r,
\]
proving \eqref{eq:weightsineq}.

In order to prove \eqref{eq:weightsint}, we fix $z' = \spnt{\theta'}{\varphi'} \in \sfera$ and split the integral in the left-hand side of \eqref{eq:weightsint}
into the sum $\sum_{j=1}^3 \cI_j$, where
\[
\cI_j=
\int_{\cS_j} (1+\dist(z,z')/r)^{-\beta} (1+\weight_r(z,z'))^{-\alpha} \,d\meas(z)
\]
and
\begin{align*}
\cS_1&=\left\{\spnt{\theta}{\varphi} \in \sfera \tc  {|\varphi-\varphi'|}^{1/2} \leq \frac{|\varphi-\varphi'|}{\max\{|\tan \theta|,|\tan \theta'|\}}\right\},\\
\cS_2&=\left\{\spnt{\theta}{\varphi}\in \sfera \setminus \cS_1 \tc |\theta'| \leq |\theta|/2 \right\},\\
\cS_3&=\left\{\spnt{\theta}{\varphi}\in \sfera \setminus \cS_1 \tc |\theta|/2 < |\theta'| \right\}.
\end{align*}

In order to estimate
$\cI_1$, we
decompose $\beta=\beta_1+\beta_2$, with $\beta_1>1-\alpha$ and $\beta_2>2$.
Then, by \eqref{eq:rho-dist},
\begin{align*}
\cI_1
&\leq
\int_{\cS_1} (1+\dist(z,z')/r)^{-\beta_2} (1+\weight_r(z,z'))^{-\alpha-\beta_1} \,d\meas(z) \\
&\leq
\int_{0}^{2\pi} (1+|\varphi-\varphi'|^{1/2}/r)^{-\beta_2} \,d\varphi
\int_{-\pi/2}^{\pi/2} \left(1+\frac{|\theta|}{\max\{r,|\theta'|\}}\right)^{-\alpha-\beta_1} \,d\theta \\
&\lesssim r^{2}\, \max\{r,|\theta'|\}
\simeq V(z',r).
\end{align*}

As for $\cI_2$, we write $\beta=\tilde\beta_1+ \tilde\beta_2$, with both $\tilde\beta_1$ and $\tilde\beta_2$ larger than $1$, so, again by \eqref{eq:rho-dist},
\[\begin{split}
\cI_2
&\lesssim \int_{\cS_2} \left(1+\frac{|\theta-\theta'|}{r}+\frac{|\varphi-\varphi'|}{r |\tan \theta|}\right)^{-\beta} \,d\meas(\spnt{\theta}{\varphi}) \\
&\lesssim \int_{-\pi/2}^{\pi/2} \left(1+\frac{|\theta|}{r}\right)^{-\tilde\beta_1}\int_0^{2\pi} \left(1+\frac{|\varphi-\varphi'|}{r |\tan \theta|}\right)^{-\tilde\beta_2} \,d\varphi \,\cos\theta \,d\theta \\
&\lesssim r \int_{-\pi/2}^{\pi/2}  \left(1+\frac{|\theta|}{r}\right)^{-\tilde\beta_1}  \,|\sin\theta| \,d\theta\\
&\lesssim r \int_{-\pi/2}^{\pi/2}  \left(1+\frac{|\theta|}{r}\right)^{-\tilde\beta_1}  \, |\theta| \,d\theta\lesssim r^3  \lesssim V(z',r).
\end{split}\]
where we used the fact that $|\tan \theta| \geq |\tan \theta'|$ and $|\theta-\theta'| \simeq |\theta|$, since $|\theta'| \leq |\theta|/2$.

To estimate $\cI_3$, in the case $|\theta'|\leq \pi/4$ we decompose $\beta$  as above and get
\[\begin{split}
\cI_3
&\lesssim \int_{\cS_3}  \left(1+\frac{|\theta-\theta'|}{r}\right)^{-\tilde\beta_1}  \left(1+ \frac{|\varphi-\varphi'|}{r |\tan \theta'|}\right)^{-\tilde\beta_2} \,d\meas(\spnt{\theta}{\varphi}) \\
&\lesssim r |\tan \theta'| \int_{-\pi/2}^{\pi/2} \left(1+\frac{|\theta-\theta'|}{r}\right)^{-\tilde\beta_1}  \,d\theta\\
&\simeq r^2 | \theta'| \lesssim V(z',r),
\end{split}\]
where we used the fact that $|\tan \theta| \lesssim |\tan \theta'| \simeq |\theta'|$, since $|\theta|/2 \leq |\theta'| \leq \pi/4$.

Finally, we are left with the estimate of $\cI_3$ in the case $|\theta'|>\pi/4$.
Here we can use \eqref{eq:rho-dist-far} to conclude that
\[
\cI_3  \lesssim \int_{\cS_3} (1+\dist_R(z,z')/r)^{-\beta} \,d\meas(z) \lesssim r^2 \simeq V(z',r),
\]
since $r< 1$ and $\beta>2$ (cf.\ \cite[Lemma 4.4]{DOS}).
\end{proof}

We now have  all the ingredients to prove our main result, Theorem \ref{thm:mainmain}.
Actually, given the above estimates, the proof reduces to standard arguments, that can be found in several places in the literature (see, in particular, \cite{He,CoS,DOS,MSi}). On the other hand, due to the particular combination of features of the manifold and operator under consideration (e.g., $\opL$ has a discrete spectrum, it is not group-invariant and the associated topological and homogeneous dimensions differ), there seems to be no existing result that can be immediately applied to $\opL$. Therefore, for the reader's convenience, we give a sketch of the proof, mainly following the scheme described in \cite[Sections 3 and 4]{M}, to which we refer for additional details.

Note that the $L^p$-boundedness in Theorem \ref{thm:mainmain}\ref{en:mainmain_mh} is clearly a consequence of the weak type $(1,1)$ bound by Marcinkiewicz interpolation and taking adjoints (indeed the right-hand side of \eqref{eq:mainmain_mh_wt11} is invariant under conjugation of $F$ and majorizes $\sup_{\lambda \geq 0} |F(\lambda)|$ by Sobolev's embedding). The change of variables $\lambda \mapsto \sqrt{\lambda}$ on the spectral side then shows that Theorem \ref{thm:mainmain} reduces to the following statement.

\begin{theorem}\label{thm:main}
Let $d = 2$ be the topological dimension of $\sfera$. Let $s > d/2$.
\begin{enumerate}[label=(\roman*)]
\item\label{en:main_compact} For all continuous functions $F : \RR \to \CC$ supported in $[1/2,1]$,
\[
\sup_{t>0} \| F(t \sqrt{\opL}) \|_{1 \to 1} \leq C_s \, \|F\|_{\Sob{2}{s}}.
\]
\item\label{en:main_mh} For all bounded Borel functions $F : \RR \to \CC$ such that $F|_{(0,\infty)}$ is continuous,
\begin{equation}\label{eq:main_mh_wt11}
\|F(\sqrt{\opL}) \|_{L^1 \to L^{1,\infty}} \leq C_s \, \sup_{t \geq 0} \| F(t \cdot) \, \eta \|_{\Sob{2}{s}}.
\end{equation}
\item\label{en:main_br} Let $p \in [1,\infty]$. If $\delta > (d-1)|1/2-1/p|$, then the Bochner--Riesz means $(1-t\opL)_+^\delta$ of order $\delta$ are bounded on $L^p(\sfera)$ uniformly in $t \in (0,\infty)$.
\end{enumerate}
\end{theorem}

\begin{proof}
By Corollary \ref{cor:heatkernel} and \cite[Theorems 1 and 4]{S}, the heat kernel $\Kern_{\exp(-t\opL)}$ has Gaussian-type heat kernel bounds: there exists $b \in (0,\infty)$ such that
\[
|\Kern_{\exp(-t\opL)}(z,z')| \leq C V(z',t^{1/2})^{-1} \exp(-b\dist(z,z')^2/t)
\]
for all $t \in (0,\infty)$ and $z,z' \in \sfera$.

Hence we can apply \cite[Theorem 6.1]{M} to obtain that, for all $\epsilon > 0$, all $\beta \geq 0$, all $R \in (0,\infty)$ and all $F : \RR \to \CC$ supported in $[-R^2,R^2]$,
\begin{align}
\vvvert \Kern_{F(\opL)} \vvvert_{2,\beta,0,R^{-1}} &\leq C_{\beta,\epsilon} \|F(R^2 \cdot)\|_{\Sob{\infty}{\beta+\epsilon}}, \label{eq:nswl2bd} \\
\| F(\opL) \|_{L^1(\sfera) \to L^1(\sfera)} &\leq C_{\epsilon} \|F(R^2 \cdot)\|_{\Sob{\infty}{Q/2+\epsilon}}  \label{eq:nsl1bd},
\end{align}

Set $A_t = \exp(-t^2\opL)$ if $t \in [0,\infty)$ and $A_t = 0$ if $t = \infty$. From \eqref{eq:nswl2bd} we deduce that, for all $t \in [0,\infty]$, all $\epsilon > 0$, all $\beta \geq 0$, all $R \in (0,\infty)$ and all $F : \RR \to \CC$ supported in $[R/16,R]$,
\[
\vvvert \Kern_{F(\sqrt{\opL}) (1-A_t)} \vvvert_{2,\beta,0,R^{-1}} \leq C_{\beta,\epsilon} \|F(R \cdot)\|_{\Sob{\infty}{\beta+\epsilon}} \min \{1,(Rt)^2 \}.
\]
Hence, if $\xi \in C_c((-1/16,1/16))$ is the mollifier defined as in \cite[eq.\ (18)]{M}, by Young's inequality we obtain that, for all $t \in [0,\infty]$, all $\epsilon > 0$, all $\beta \geq 0$, all $R \in (0,\infty)$ and all $F : \RR \to \CC$ supported in $[R/8,7R/8]$,
\[
\vvvert \Kern_{(\xi* F)(\sqrt{\opL})(1-A_t)} \vvvert_{2,\beta,0,R^{-1}} \leq C_{\beta,\epsilon} \|F(R \cdot)\|_{\Sob{\infty}{\beta+\epsilon}} \min \{1,(Rt)^2 \}.
\]
In particular, by \eqref{eq:weightsineq} and Sobolev's embedding, for all $t \in [0,\infty]$, all $\epsilon > 0$, all $\alpha,\beta \geq 0$, all $N \in \NN \setminus \{0\}$ and all $F : \RR \to \CC$ supported in $[N/8,7N/8]$,
\begin{equation}\label{eq:we1}
\vvvert \Kern_{(\xi*F)(\sqrt{\opL})(1-A_t)} \vvvert_{2,\beta,\alpha,N^{-1}} \leq C_{\alpha,\beta,\epsilon} \|F(N \cdot)\|_{\Sob{2}{\beta+\alpha+1/2+\epsilon}} \min \{1,(Nt)^2 \}.
\end{equation}

On the other hand, by Proposition \ref{prp:weighted_plancherel}, for all $t \in [0,\infty]$, all $\alpha \in [0,1/2)$, all $N \in \NN \setminus \{0\}$ and all $F : \RR \to \CC$ supported in $[N/16,N]$,
\[
\vvvert \Kern_{F(\sqrt{\opL}) (1-A_t)} \vvvert_{2,0,\alpha,N^{-1}} \leq C_{\alpha} \|F(N \cdot)\|_{N,2} \min \{1,(Nt)^2 \}.
\]
Hence, by \cite[eq.\ (4.9)]{DOS}, for all $t \in [0,\infty]$, all $\alpha \in [0,1/2)$, all $N \in \NN \setminus \{0\}$ and all $F : \RR \to \CC$ supported in $[N/8,7N/8]$,
\begin{equation}\label{eq:we2}
\vvvert \Kern_{(\xi*F)(\sqrt{\opL})(1-A_t)} \vvvert_{2,0,\alpha,N^{-1}} \leq C_{\alpha} \|F(N \cdot)\|_{L^2} \min \{1,(Nt)^2 \}.
\end{equation}

Interpolation of \eqref{eq:we1} and \eqref{eq:we2} gives that, for all $t \in [0,\infty]$, all $\epsilon > 0$, all $\alpha \in [0,1/2)$, all $\beta \geq 0$, all $N \in \NN \setminus \{0\}$ and all $F : \RR \to \CC$ supported in $[N/4,3N/4]$,
\[
\vvvert \Kern_{(\xi*F)(\sqrt{\opL})(1-A_t)} \vvvert_{2,\beta,\alpha,N^{-1}} \leq C_{\alpha,\beta,\epsilon} \|F(N \cdot)\|_{\Sob{2}{\beta+\epsilon}} \min \{1,(Nt)^2 \}.
\]
By \eqref{eq:weightsint} and H\"older's inequality we then deduce that, for all $r \in [0,\infty)$, all $t \in [0,\infty]$, all $s > d/2$, all $\epsilon \in [0,s-d/2)$,  all $N \in \NN \setminus \{0\}$ and all $F : \RR \to \CC$ supported in $[N/4,3N/4]$,
\begin{equation}\label{eq:mollified_l1}
\begin{split}
\esssup_{z' \in \sfera} &\int_{\sfera \setminus B(z',r)} |\Kern_{(\xi*F)(\sqrt{\opL})(1-A_t)}(z,z')| \,d\meas(z) \\
 &\leq (1+Nr)^{-\epsilon} \vvvert \Kern_{(\xi*F)(\sqrt{\opL})(1-A_t)} \vvvert_{1,\epsilon,0,N^{-1}} \\
 &\leq C_{s,\epsilon} (1+Nr)^{-\epsilon} \vvvert \Kern_{(\xi*F)(\sqrt{\opL})(1-A_t)} \vvvert_{2,\beta,\alpha,N^{-1}} \\
 &\leq C_{s,\epsilon} (1+Nr)^{-\epsilon} \|F(N \cdot)\|_{\Sob{2}{s}} \min \{1,(Nt)^2 \},
\end{split}
\end{equation}
where $\alpha \in [0,1/2)$ and $\beta \in [0,\infty)$ were chosen so that $\beta < s$ and $\alpha +\beta -\epsilon > Q/2$.

On the other hand, if $D$ is the $\dist$-diameter of $\sfera$, by \eqref{eq:weightsint}, H\"older's inequality, Proposition \ref{prp:weighted_plancherel} and \cite[Proposition 4.6]{DOS},  for all $s > d/2$, all $\epsilon \in [0,\min\{s-d/2,1/2\})$, all  $N \in \NN \setminus \{0\}$ and all $F : \RR \to \CC$ supported in $[N/4,3N/4]$,
\begin{equation}\label{eq:nomollified_l1}
\begin{split}
\| (F-\xi*F)(\sqrt{\opL}) \|_{1 \to 1} &= \vvvert \Kern_{(F-\xi*F)(\sqrt{\opL})} \vvvert_{1,0,0,N^{-1}} \\
&\leq C_{s,\epsilon} \vvvert \Kern_{(F-\xi*F)(\sqrt{\opL})} \vvvert_{2,\beta,\alpha,N^{-1}} \\
&\leq C_{s,\epsilon} (1+ND)^\beta \vvvert \Kern_{(F-\xi*F)(\sqrt{\opL})} \vvvert_{2,0,\alpha,N^{-1}} \\
&\leq C_{s,\epsilon} N^\beta \| (F-\xi*F)(N \cdot) \|_{N,2} \\
&\leq C_{s,\epsilon} N^{-\epsilon} \|F(N \cdot)\|_{\Sob{2}{\epsilon+\beta}} \\
&\leq C_{s,\epsilon} N^{-\epsilon} \|F(N \cdot)\|_{\Sob{2}{s}},
\end{split}
\end{equation}
where $\alpha \in [0,1/2)$ and $\beta \in (d/2,\infty)$ were chosen so that $\epsilon + \beta < \min\{Q/2,s\}$ and $\alpha +\beta > Q/2$.

Finally, observe that the only eigenvalue of $\sqrt{\opL}$ lying in $(-\infty,1)$ is $0$. Hence, for all $F : \RR \to \CC$ supported in $(-\infty,1)$,
\begin{equation}\label{eq:trivial_l1}
\|\Kern_{F(\sqrt{\opL})}\|_{1 \to 1} = C |F(0)|.
\end{equation}

Combining \eqref{eq:mollified_l1} (applied with $t=\infty$, and $\epsilon = r=0$) and \eqref{eq:nomollified_l1} (applied with $\epsilon=0$) gives in particular that, for all $s > d/2$, all  $N \in \NN \setminus \{0\}$ and all $F : \RR \to \CC$ supported in $[N/4,3N/4]$,
\begin{equation}\label{eq:l1_compact}
\| F(\sqrt{\opL}) \|_{1 \to 1} \leq C_s \| F(N \cdot) \|_{\Sob{2}{s}}.
\end{equation}
This estimate, together with \eqref{eq:trivial_l1}, easily gives part \ref{en:main_compact}.

As for part \ref{en:main_mh},
since the right-hand side of \eqref{eq:main_mh_wt11} is essentially independent of the cut-off function $\eta$, we may assume that $\supp \eta \subseteq (1/4,1)$ and $\sum_{k \in \ZZ} \eta(2^k \cdot) = 1$ on $(0,\infty)$.
Then, by the use of the dyadic decomposition $F = \sum_{k \in \NN} \eta(2^{-k} \cdot) F $ and an application of \cite[Theorem 1]{DMc}, from \eqref{eq:mollified_l1} and \eqref{eq:nomollified_l1} we obtain that, for all $F : \RR \to \CC$ supported in $[1/2,\infty)$,
\begin{equation}\label{eq:l1weak}
\| F(\sqrt{\opL}) \|_{L^1 \to L^{1,\infty}} \leq C_s \sup_{k \in \NN} \| \eta \, F(2^k \cdot) \|_{\Sob{2}{s}}.
\end{equation}
Via a partition of unity subordinated to $\{(1/2,\infty),(-\infty,1)\}$, we can now combine \eqref{eq:l1weak} and \eqref{eq:trivial_l1} and obtain part \ref{en:main_mh}.

It remains to prove part \ref{en:main_br}. Define $\phi_w : \RR \to \CC$ by $\phi_w(\lambda) = (1-\lambda)_+^w$, and note that $\sup_{\lambda \in [0,\infty)} |\phi_w(\lambda)| \leq 1$ whenever $w \in \CC$, $\Re w \geq 0$. Hence $\phi_w(t\opL) = (1-t\opL)_+^w$ is bounded on $L^2(\sfera)$ with at most unit norm for all $t \in (0,\infty)$ and $w \in \CC$, $\Re w \geq 0$. Therefore, by complex interpolation \cite[Theorem V.4.1]{Stein-Weiss}, it is enough to prove that $\phi_w(t\opL)$ is bounded on $L^1(\sfera)$ uniformly in $t \in (0,\infty)$ whenever $w \in \CC$ and $\Re w > (d-1)/2$, with a bound that grows at most polynomially in $\Im w$ for any fixed value of $\Re w$.

This is easily obtained by splitting $\phi_w = \phi_w^{-} + \phi^0_w + \phi^+_w$ through a smooth partition of unity subordinated to $\{(-\infty,-1/2),(-1,1),(1/2,\infty)\}$. Indeed $\phi_w^-(t\opL) = 0$, since $\opL$ has nonnegative spectrum. Moreover $\phi_0^w \in C^\infty_c(\RR)$, with derivatives bounded polynomially in $|w|$, so the required bound for $\phi_w^0(t\opL)$ follows from \eqref{eq:nsl1bd}. Finally, $\phi_w^+ \in \Sob{2}{s}(\RR)$ whenever $\Re w > s-1/2$, so the required bound for $\phi_w^+(t\opL)$ follows from part \ref{en:main_compact}.
\end{proof}

\acknowledgments



\begin{thebibliography}{MMMM}

\bibitem[ASte]{ASt}
M. Abramowitz and I. A. Stegun,
\emph{Handbook of mathematical functions: with formulas, graphs, and mathematical tables},
Dover Publications, New York, 2012.


\bibitem[ACMM]{ACMMu}
J. Ahrens, M. G. Cowling, A. Martini, and D. M\"uller,
Quaternionic spherical harmonics and a sharp multiplier theorem on quaternionic spheres,
\emph{Math. Z.} (to appear),
\texttt{arXiv:1612.04802}.


\bibitem[AsWa]{AW}
R. Askey and S. Wainger,
Mean convergence of expansions in Laguerre and Hermite series,
\emph{Amer. J. Math.}
\textbf{87} (1965),
695--708.


\bibitem[BeRi]{BellaicheRisler}
A. Bella\"iche and J.-J. Risler,
\emph{Sub-Riemannian geometry},
Progress in Mathematics, Vol.\ 144, Birkh\"auser, Basel, 1996.


\bibitem[BFI1]{BFIuno}
W. Bauer, K. Furutani, and C. Iwasaki,
Spectral analysis and geometry of sub-Laplacian and related Grushin-type operators,
\emph{Partial differential equations and spectral theory},
Oper. Theory Adv. Appl., 211, Birkh\"auser/Springer Basel AG, Basel, 2011, 183--290.


\bibitem[BFI2]{BFI}
W. Bauer, K. Furutani, and C. Iwasaki,
{Fundamental solution of a higher step Grushin type operator},
\emph{Adv. Math.}
\textbf{271} (2015),
188--234.


\bibitem[BoL]{BoL}
U. Boscain and C. Laurent,
The Laplace-Beltrami operator in almost-Riemannian geometry,
\emph{Ann. Inst. Fourier (Grenoble)}
\textbf{63} (2013),
1739--1770.


\bibitem[BoPSe]{BPS}
U. Boscain, D. Prandi, and M. Seri,
Spectral analysis and the Aharonov--Bohm effect on certain almost-Riemannian manifolds,
\emph{Comm. Partial Differential Equations}
\textbf{41} (2016),
32--50.


\bibitem[BoyD]{BoydDunster}
W. G. C. Boyd and T. M. Dunster,
Uniform asymptotic solutions of a class of second-order linear differential equations having a turning point and a regular singularity, with an application to Legendre functions,
\emph{SIAM J. Math. Anal.}
\textbf{17} (1986),
422--450.


\bibitem[BDWZ]{BDWZ}
N. Burq, S. Dyatlov, R. Ward, and M. Zworski,
Weighted eigenfunction estimates with applications to compressed sensing,
\emph{SIAM J. Math. Anal.}
\textbf{44} (2012),
3481--3501.


\bibitem[CaCh]{CalinChang}
O. Calin and D.-C. Chang,
\emph{Sub-Riemannian geometry},
Encyclopedia of Mathematics and its Applications, Vol. 126, Cambridge University Press, 2009.



\bibitem[CCMS]{CCMS}
V. Casarino, M. G. Cowling, A. Martini, and A. Sikora,
Spectral multipliers for the Kohn Laplacian on forms on the sphere in $\mathbb{C}^n$,
\emph{J. Geom. Anal.} \textbf{27} (2017), 3302--3338.



\bibitem[COu]{CO}
P. Chen and E. M. Ouhabaz,
Weighted restriction type estimates for Grushin operators and application to spectral multipliers and Bochner--Riesz summability,
\emph{Math. Z.}
\textbf{282} (2016),
663--678.


\bibitem[CSi]{CS}
P. Chen and A. Sikora,
Sharp spectral multipliers for a new class of Grushin type operators,
\emph{J. Fourier Anal. Appl.}
\textbf{19} (2013),
1274--1293.



\bibitem[CoKSi]{CoKS}
M. G. Cowling, O. Klima, and A. Sikora,
Spectral multipliers for the Kohn sublaplacian on the sphere in $\mathbb{C}^n$,
\emph{Trans. Amer. Math. Soc.}
\textbf{363} (2011),
611--631.


\bibitem[CoM]{CM}
M. G. Cowling and A. Martini,
Sub-Finsler geometry and finite propagation speed,
\emph{Trends in harmonic analysis},
Springer INdAM Series, Springer, Milan, 2013,
147--205.


\bibitem[CoSi]{CoS}
M. G. Cowling and A. Sikora,
A spectral multiplier theorem for a sublaplacian on $\rm SU(2)$,
\emph{Math. Z.}
\textbf{238} (2001),
1--36.


\bibitem[DMc]{DMc}
X. T. Duong and A. McIntosh,
Singular integral operators with non-smooth kernels on irregular domains,
\emph{Rev. Mat. Iberoam.}
\textbf{15} (1999),
233--265.


\bibitem[DOSi]{DOS}
X. T. Duong, E. M. Ouhabaz, and A. Sikora,
Plancherel-type estimates and sharp spectral multipliers,
\emph{J. Funct. Anal.}
\textbf{196} (2002),
443--485.


\bibitem[DzJ]{DzJ}
J. Dziuba\'nski and K. Jotsaroop,
On Hardy and BMO spaces for Grushin operator,
\emph{J. Fourier Anal. Appl.}
\textbf{22} (2016),
954--995.


\bibitem[EMOT]{EMOT}
A. Erd\'elyi, W. Magnus, F. Oberhettinger, and F. G. Tricomi,
\emph{Higher transcendental functions. Vol. II},
Robert E. Krieger Publishing Co. Inc., Melbourne, Fla., 1981.


\bibitem[FePh]{FePh}
C. Fefferman and D. H. Phong,
Subelliptic eigenvalue problems,
\emph{Conference on harmonic analysis in honor of Antoni Zygmund (Chicago, Ill., 1981)},
Wadsworth, Belmont, Ca., 1983, 590--606.


\bibitem[FSab]{Franck}
R. L. Frank and J. Sabin,
Spectral cluster bounds for orthonormal systems and oscillatory integral operators in Schatten spaces,
\emph{Adv. Math.} \textbf{317} (2017),  157--192.

\bibitem[Ga]{Garofalo}
N. Garofalo,
Unique continuation for a class of elliptic operators which degenerate on a manifold of arbitrary codimension,
\emph{J. Differential Equations}
\textbf{104} (1993),
117--146.


\bibitem[G]{Grushin}
V. V. Grushin,
On a class of hypoelliptic operators,
\emph{Math. USSR  Sb.}
\textbf{12} (1970),
458--476.


\bibitem[HSc]{Haagerup}
U. Haagerup and H. Schlichtkrull,
Inequalities for Jacobi polynomials,
\emph{Ramanujan J.}
\textbf{33} (2014),
227--246.


\bibitem[Ha]{Han}
X. Han,
Spherical harmonics with maximal $L^p$ ($2 < p \leq 6$) norm growth,
\emph{J. Geom. Anal.}
\textbf{26} (2016),
378--398.


\bibitem[He1]{He_heis}
W. Hebisch,
Multiplier theorem on generalized Heisenberg groups,
\emph{Coll. Math.}
\textbf{65} (1993),
231--239.


\bibitem[He2]{He}
W. Hebisch,
Functional calculus for slowly decaying kernels,
preprint (1995),
\texttt{http://www.math.uni.wroc.pl/{\textasciitilde}hebisch/}.


\bibitem[JSaTh]{JST}
K. Jotsaroop, P. K. Sanjay, and S. Thangavelu,
Riesz transforms and multipliers for the Grushin operator,
\emph{J. Anal. Math.},
\textbf{119} (2013),
255--273.


\bibitem[JTh]{JT}
K. Jotsaroop and S. Thangavelu,
$L^p$ estimates for the wave equation associated to the Grushin operator,
\emph{Ann. Sc. Norm. Super. Pisa Cl. Sci. (5)}
\textbf{13} (2014),
775--794.



\bibitem[KeStT]{KST}
C. E. Kenig, R. J. Stanton and P. A. Tomas,
Divergence of eigenfunction expansions,
\emph{J. Funct. Anal.}
\textbf{46} (1982),
28--44.


\bibitem[Kr]{Kra}
I. Krasikov,
On the Erd\'elyi--Magnus--Nevai conjecture for Jacobi polynomials,
\emph{Constr. Approx.}
\textbf{28} (2008),
113--125.


\bibitem[LRuY]{LRY}
A. Laptev, M. Ruzhansky, and N. Yessirkegenov,
Hardy inequalities for Landau Hamiltonian and for Baouendi-Grushin operator with Aharonov-Bohm type magnetic field. Part I,
preprint (2017),
\texttt{arXiv:1705.00062}.


\bibitem[M]{M}
A. Martini,
Joint functional calculi and a sharp multiplier theorem for the Kohn Laplacian on spheres,
\emph{Math. Z.}
\textbf{286} (2017), 1539--1574.


\bibitem[MM\"u1]{MMu}
A. Martini and D. M\"uller,
A sharp multiplier theorem for Grushin operators in arbitrary dimensions,
\emph{Rev. Mat. Iberoam.}
\textbf{30} (2014),
1265--1280.


\bibitem[MM\"u2]{MMuGAFA}
A. Martini and D. M\"uller,
Spectral multipliers on 2-step groups: topological versus homogeneous dimension,
\emph{Geom. Funct. Anal.}
\textbf{26} (2016),
680--702.


\bibitem[MSi]{MSi}
A. Martini and A. Sikora,
Weighted Plancherel estimates and sharp spectral multipliers for the Grushin operators,
\emph{Math. Res. Lett.}
\textbf{19} (2012),
1075--1088.


\bibitem[Me]{Melrose}
R. B. Melrose,
Propagation for the wave group of a positive subelliptic second-order differential operator,
\emph{Hyperbolic equations and related topics (Katata/Kyoto, 1984)},
Academic Press, Boston, Ma., 1986,
181--192.


\bibitem[Mey]{Meyer}
R. Meyer,
\emph{$L^p$-estimates for the wave equation associated to the Gru\v{s}in operator},
PhD Dissertation, Christian-Albrechts-Universit\"at zu Kiel, 2006,
\texttt{arXiv:0709.2188}.


\bibitem[Mi]{Mi}
B. S. Mitjagin,
Divergenz von {S}pektralentwicklungen in $L^p$-R\"aumen,
\emph{Linear operators and approximation, II (Proc. Conf., Oberwolfach Math. Res. Inst., Oberwolfach, 1974)},
Internat. Ser. Numer. Math., Vol. 25, Birkh\"auser, Berlin, 1974, 521--530.


\bibitem[Mo]{Montgomery}
R. Montgomery,
\emph{A tour of subriemannian geometries, their geodesics and applications},
Mathematical Surveys and Monographs, Vol.\ 91, American Mathematical Society, Providence, RI, 2002.



\bibitem[MuSp]{MuSp}
M. E. Muldoon and R. Spigler,
Some remarks on zeros of cylinder functions,
\emph{SIAM J. Math. Anal.}
\textbf{15} (1984),
1231--1233.


\bibitem[M\"uS]{MuSt}
D. M\"uller and E. M. Stein,
On spectral multipliers for Heisenberg and related groups,
\emph{J. Math. Pures Appl.}
\textbf{73} (1994),
413--440.


\bibitem[O1]{Olver75}
F. W. J. Olver,
Second order linear differential equations with two turning points,
\emph{Philos. Trans. Roy. Soc. London Ser. A}
\textbf{278} (1975),
137--174.


\bibitem[O2]{Olver-libro}
F. W. J. Olver,
\emph{Asymptotics and Special Functions},
Academic Press, New York, 1974.


\bibitem[O3]{Olver}
F. W. J. Olver,
Legendre functions with both parameters large,
\emph{Philos. Trans. Roy. Soc. London Ser. A}
\textbf{278} (1975),
175--185.


\bibitem[Pe]{Pesenson}
I. Pesenson,
Parseval space-frequency localized frames on sub-Riemannian compact homogeneous manifolds,
\emph{Frames and Other Bases in Abstract and Function Spaces. Novel Methods in Harmonic Analysis, Volume 1}, Applied and and Numerical Harmonic Analysis, Birkh\"auser, Basel, 2017, 413--433.

\bibitem[RWar]{Ward}
H. Rauhut and R. Ward,
Sparse recovery for spherical harmonic expansions,
\emph{SampTA 2011 Conference Proceedings},
2011.


\bibitem[RoSi]{RS}
D. W. Robinson and A. Sikora,
Analysis of degenerate elliptic operators of Gru\v{s}in type,
\emph{Math. Z.}
\textbf{260} (2008),
475--508.


\bibitem[SeeSo]{SeeSo}
A. Seeger and C. D. Sogge,
On the boundedness of functions of (pseudo-) differential operators on compact manifolds,
\emph{Duke Math. J.}
\textbf{59} (1989),
709--736.


\bibitem[Si]{S}
A. Sikora,
Riesz transform, Gaussian bounds and the method of wave equation,
\emph{Math. Z.}
\textbf{247} (2004),
643--662.


\bibitem[So1]{Sogge1}
C. D. Sogge,
Oscillatory integrals and spherical harmonics,
\emph{Duke Math. J.}
\textbf{53} (1986),
43--65.

\bibitem[So2]{Sogge3}
C. D. Sogge,
\emph{Hangzhou lectures on eigenfunctions of the Laplacian},
Annals of Mathematics Studies, Vol. 188,
Princeton University Press, Princeton, NJ, 2014.

\bibitem[So3]{Sogge2}
C. D. Sogge,
Problems related to the concentration of eigenfunctions,
\emph{Journ\'ees ``\'Equations aux D\'eriv\'ees Partielles''} (2015), Expos\'e IX.

\bibitem[So4]{Sogge4}
C. D. Sogge,
\emph{Fourier Integrals in Classical Analysis},
Second Edition, Cambridge University Press, Cambridge, 2017.

\bibitem[SW1]{StW_interp}
E. M. Stein and G. L. Weiss,
Interpolation of operators with change of measure,
\emph{Trans. Amer. Math. Soc.}
\textbf{87} (1958),
159--172.


\bibitem[SW2]{Stein-Weiss}
E. M. Stein and G. L. Weiss,
\emph{Introduction to Fourier analysis on Euclidean spaces},
Princeton University Press, Princeton, NJ, 1971.


\bibitem[Sz]{Szego}
G. Szeg\H{o},
\emph{Orthogonal polynomials},
Amer. Math. Soc. Colloq. Publ., vol. 23, Amererican Mathematical Society, 4th ed.
Providence, RI, 1974.


\bibitem[Te]{Te}
N. M. Temme,
\emph{Special functions},
John Wiley \& Sons, Inc., New York, 1996.


\bibitem[Th]{Th}
S. Thangavelu,
\emph{Lectures on Hermite and Laguerre expansions},
Mathematical Notes, Vol. 42,
Princeton University Press,
Princeton, NJ, 1993.


\bibitem[Ze1]{Ze1}
S. Zelditch,
Local and global analysis of eigenfunctions on Riemannian manifolds,
\emph{Handbook of geometric analysis. No. 1},
Adv. Lect. Math. (ALM), Vol. 7,
International Press, Somerville, Ma., 2008, 545--658.

\bibitem[Ze2]{Ze2}
S. Zelditch,
\emph{Park City lectures on Eigenfunctions} (2013),
\texttt{arXiv:1310.7888}.

\bibitem[Zw]{Zw}
D. Zwillinger,
\emph{Standard mathematical tables and formulae}, 31st ed.,
Chapman \& Hall/CRC, Boca Raton, Fl., 2003.


\end{thebibliography}
\end{document}